\documentclass[reqno,11pt]{amsart}
\usepackage{amsmath,amsfonts,mathrsfs,amssymb,graphicx,color}
\def\tr{\mathop{\rm tr}\nolimits}

\newcommand{\Z}{{\mathbb Z}}
\newcommand{\R}{{\mathbb R}}
\newcommand{\A}{{\mathcal A}}
\newcommand{\B}{{\mathcal B}}
\newcommand{\Q}{{\mathbb Q}}

\newcommand{\F}{{{\mathcal F}}}
 \newcommand{\N}{{\mathbb N}}
 \newcommand{\NN}{{\mathcal N}}

 \allowdisplaybreaks[4]
\newtheorem{thm}{Theorem}[section]

\newtheorem{prop}[thm]{Proposition}
\newtheorem{rem}[thm]{Remark}
\newtheorem{remark}{Remark}[section]
\newtheorem{lem}[thm]{Lemma}
\newtheorem{defi}[thm]{Definition}

 \numberwithin{equation}{section}

\begin{document}
\title{Exact-dimensional property of  density of states measure  of Sturm Hamiltonian  }

\author{Yanhui QU}
\address[Y.H. QU]{Department of  Mathematical Science, Tsinghua University, Beijing 100084, P. R. China}
\email{yhqu@math.tsinghua.edu.cn}

\begin{abstract}

For frequency $\alpha$ of bounded type and coupling $\lambda>20$, we show that the density of states measure $\NN_{\alpha,\lambda}$ of the related Sturm Hamiltonian is exact upper and lower dimensional, however, in general it is not exact-dimensional.

\end{abstract}

\maketitle

\section{Introduction}

Since the work \cite{BIST},  the Sturm Hamiltonian has been extensively studied as a typical model of quasi-periodic Schr\"odinger operator.  
The Sturm Hamiltonian is a  bounded self-adjoint operator on $\ell^2(\Z),$ defined by 
$$
(H_{\alpha,\lambda,\theta}\psi)_n:=\psi_{n-1}+\psi_{n+1}+\lambda\chi_{[1-\alpha,1)}(n\alpha+\theta\pmod 1) \psi_n,
$$
where $\alpha\in [0,1]\setminus \Q$,
$\lambda>0$  and  $\theta\in[0,1)$. $\alpha, \lambda, \theta$ are
called the  {\it frequency, coupling} and {\it phase}, respectively.  It is well-known that the spectrum and the density of states measure (DOS)  of Sturm Hamiltonian
 are independent of   $\theta$ and   we  denote them   by $\Sigma_{\alpha,\lambda}$ and $\NN_{\alpha,\lambda}$, respectively (see \cite{BIST,CL} for detail).    
  The fractal dimensions of the spectrum have been studied by many authors, see \cite{Da} for a detailed review.   In this paper, we  focus on  the dimensional  properties, especially the exact-dimensional properties   of    $\NN_{\alpha,\lambda}$. Let us recall the related definitions.

   Assume $\alpha\in [0,1]\setminus \Q$ has continued fraction expansion $\alpha=[0;a_1,a_2, \cdots]$ with $a_n\in \N.$ If $\{a_n: n\ge 1\}$ is bounded, $\alpha$ is called of {\it bounded type}. If $a_n=\kappa$ for $n\ge N$, $\alpha$ is called of {\it eventually constant type}. $\alpha_\kappa:=[0;\kappa,\kappa,
\cdots]$ is called of {\it constant type}. The most famous frequency of constant type is the inverse of golden number $\alpha_1=[0;1,1,\cdots]=(\sqrt{5}-1)/2.$ The Sturm Hamiltonian $H_{\alpha_1,\lambda,\theta}$ is called {\it Fibonacci } Hamiltonian.

 Assume finite measure $\mu$ is defined on a compact metric space $X$.    Fix $x\in X$, we define the {\it  upper} and {\it lower } local dimensions of $\mu$ at $x$ as
$$
\overline{d}_\mu(x):=\limsup_{r\to0}\frac{\log \mu(B(x,r))}{\log r}\ \ \ \text{ and }\ \ \ \underline{d}_\mu(x):=\liminf_{r\to0}\frac{\log \mu(B(x,r))}{\log r}.
$$
If  $\overline{d}_\mu(x)=\underline{d}_\mu(x)$, we say that the {\it local  dimension} of   $\mu$ at $x$ exists and denote it by  $d_\mu(x)$. 
The Hausdorff and packing dimensions of $\mu$ are defined as
 $$
  \begin{cases}
  \dim_H\mu:=\sup\{s: \underline{d}_\mu(x)\ge s \text{ for  } \mu \text{ a.e. }x\in X\},\\
  \dim_P\mu:=\sup\{s: \overline{d}_\mu(x)\ge s \text{ for  } \mu \text{ a.e. }x\in X\}.
  \end{cases}
  $$

If there exists a constant $d$ such that $\underline{d}_\mu(x)=d\  (\overline{d}_\mu(x)=d)$ for $\mu$ a.e. $x\in X$, then necessarily $\dim_H\mu=d$ ($\dim_P\mu=d$). In this case we say that  $\mu$ is  {\it  exact lower (upper) dimensional}. If there exists a constant $d$ such that ${d}_\mu(x)=d$ for $\mu$ a.e. $x\in X$, then necessarily $\dim_H\mu=\dim_P\mu=d$. In this case we say that  $\mu$ is {\it exact-dimensional}. 

Fix $\alpha\in [0,1]\setminus \Q$ and $\lambda>0$, write 
\begin{eqnarray*}
d_H(\alpha,\lambda):=\dim_H \NN_{\alpha,\lambda}, & d_P(\alpha,\lambda):=\dim_P \NN_{\alpha,\lambda}.
\end{eqnarray*}

 The most prominent model among the Sturm Hamiltonian is the Fibonacci Hamiltonian, which was  introduced by physicists  to model the quasicrystal  system, see  \cite{KKT,OPRSS}.
  The  dimensional properties of its  DOS   have been studied in many works,  see for example  \cite{R,DG2,DG3,DG4,P}, especially the  recent work \cite{DGY}.   We summarize the results which are related to our paper as follows:  $\NN_{\alpha_1,\lambda}$ is exact-dimensional and
  \begin{equation}\label{dim-Fibonacci}
 d(\alpha_1,\lambda):=d_H(\alpha_1,\lambda)=d_P(\alpha_1,\lambda)=\frac{h_{\nu_\lambda}}{{\rm Lyap}^u\nu_\lambda},
\end{equation}
    where $\nu_\lambda$ is the measure of maximal entropy of the Fibonacci trace map $T_\lambda$ and ${\rm Lyap}^u\nu_\lambda$ is the unstable Lyapunov exponent of $\nu_\lambda.$
    $d(\alpha_1,\lambda)$ is an analytic function of $\lambda$  and     
\begin{equation}\label{dos-Fibo}
\lim_{\lambda\to0} d(\alpha_1,\lambda)=1;\ \ \lim_{\lambda\to\infty} d(\alpha_1,\lambda)\log  \lambda =-\frac{5+\sqrt{5}}{4}\log \alpha_1.
\end{equation}

   Girand \cite{Gi} and  Mei \cite{Mei} considered the frequency $\alpha$ with eventually periodic continued fraction expansion. In both papers they showed that $\lim_{\lambda\to0}d(\alpha,\lambda)=1$ and 
   $\NN_{\alpha,\lambda}$ is exact dimensional for small $\lambda$. This generalizes the results in \cite{DG3}. 
   Munger \cite{Mu}  gave estimation on the optimal H\"older exponent  of $\NN_{\alpha_\kappa,\lambda}$ and obtained   asymptotic formula for it.   For $\lambda>20$,  Qu \cite{Q} obtained the dimension formula  of $\NN_{\alpha_\kappa,\lambda}$ similar with \eqref{dim-Fibonacci}, he also showed that $\NN_{\alpha_\kappa,\lambda}$ is exact-dimensional and obtained similar asymptotic behavior  as \eqref{dos-Fibo}.  We remark that, for all works mentioned  above, the dynamical method is applicable due to the special types of the frequencies.

 Recently,  Jitomirskaya and Zhang \cite{JZ} showed that  if $\beta(\alpha)>0$, then $d_P(\alpha,\lambda)=1$, where $\beta(\alpha):=\limsup_n\frac{\log q_{n+1}}{q_n}$ and $p_n/q_n$ is the $n$-th  continued fraction approximation  of $\alpha$. They also constructed 
 specific  $\alpha$ with $\beta(\alpha)>0$ such that  
  $d_H(\alpha,\lambda)<1$  when  $\lambda>20$. In particular,  $\NN_{\alpha,\lambda}$ is not exact-dimensional for such $\alpha$. We remark that the set $\{\alpha\in [0,1]\setminus\Q: \beta(\alpha)>0\}$ has Hausdorff dimension $0$.
  
  Except Jitomirskaya-Zhang's result, almost nothing is known about the dimensional property of  DOS for general Sturm Hamiltonian.
  The main motivation of this  work is to understand the dimensional property of DOS for general Sturm Hamiltonian. As the first step, we concentrate on the frequencies of bounded type.  Let us write  
  $$
 \mathscr{B}:=\{ \alpha\in [0,1]\setminus \Q: \alpha  \text{ is of bounded type}\}.
$$

 Our main result is as follows.       
   
\begin{thm}\label{main}
(i) Assume  $\alpha\in\mathscr{B}$ and $\lambda>20$, then $\NN_{\alpha,\lambda}$ is   exact upper and lower dimensional.  As a consequence, $\NN_{\alpha,\lambda}$ is exact-dimensional if and only if $d_H(\alpha,\lambda)=d_P(\alpha,\lambda)$.

(ii) There exists $\alpha\in \mathscr{B}$ such that $\NN_{\alpha,\lambda}$ is not exact-dimensional for $\lambda$ large enough.
\end{thm}

\begin{remark}
{\rm
(i) Like the example in \cite{JZ}, our result reveals a new phenomenon, which does not occur  for frequencies with eventually periodic expansions.
Our result says  more about the regularity of the DOS: although it can fail to be exact-dimensional, it is still nice in the sense that it is  exact upper and lower dimensional.  

(ii) By \cite{LW}, $d_P(\alpha,\lambda)<1$ for $\alpha\in \mathscr{B},$ so our result is different from that in \cite{JZ}. Indeed, all the frequencies in $\mathscr{B}$ are Diophantine, while the frequencies considered  in \cite{JZ} are Liouville.   It is known that $\mathscr{B}$ has Lebesgue measure $0$, but has Hausdorff dimension $1$. Thus the size of $\mathscr{B}$ is relatively  big.

(iii) From  fractal geometry point of view, it seems quite non-trivial to construct a finite measure  such that it is exact upper and lower dimensional, nevertheless, it is not exact-dimensional. Here, we  obtain  such kind of measures naturally.
} 
\end{remark} 

 In the following, we roughly describe our idea of proof. The key notion we will introduce is the {\it Gibbs-like measure}, which is an analog of Gibbs measure in the non-dynamical setting. 
 
 It is known that $\Sigma_{\alpha,\lambda}$   can be  coded  by a symbolic space $\Omega^{(\alpha)}$ through the coding map $\pi_\alpha:\Omega^{(\alpha)}\to{\Sigma}_{\alpha,\lambda}.$ Moreover
$$
{\Sigma}_{\alpha,\lambda}=\bigcap_{n\ge 1}\bigcup_{w\in \Omega^{(\alpha)}_{n} }B_w,
$$
where $\{B_w:w\in\Omega^{(\alpha)}_{n}\}$ are intervals related to the $n$-th periodic approximation of the operator $H_{\alpha,\lambda,0}$  (see \cite{R,LQW,Q}). The symbolic space $\Omega^{(\alpha)}$ is a generalization of subshift of finite type, which is defined by a sequence of alphabets and incidence matrices.
 In general, there is no dynamic on $\Omega^{(\alpha)}$ since  the shift map is not invariant. If $\alpha$ is of bounded type, only finitely many alphabets and incidence matrices are needed to construct $\Omega^{(\alpha)}.$  
 
 One can define a metric $d_\alpha$ on $\Omega^{(\alpha)}$ by 
 $$
 d_\alpha(x,y):=|B_{x\wedge y}|.
 $$
 It can be shown that
   $\pi_\alpha: (\Omega^{(\alpha)}, d_\alpha)\to({\Sigma}_{\alpha,\lambda}, |\cdot |)$ is bi-Lipschitz.  Define 
  $$
  \mu_\alpha:=(\pi_{\alpha}^{-1})_\ast({ \NN_{\alpha,\lambda}}),
  $$
  then $\mu_\alpha$ is supported on $\Omega^{(\alpha)}$ and  has the same dimensional property with $\NN_{\alpha,\lambda}$ since $\pi_\alpha$ is bi-Lipschitz. By applying \cite{R}, $\mu_\alpha$ has the following equivalent definition:  define 
 $$
 \mu_n:=\frac{1}{\#\Omega^{(\alpha)}_{n}}\sum_{w\in \Omega^{(\alpha)}_{n}}\delta_{x_w},
 $$
  where $x_w$ is any fixed point in the cylinder $ [w]^\alpha$.
    Then $\mu_n$ converge to $\mu_\alpha$ in weak-star topology.  This definition suggests that $\mu_\alpha$ is the ``measure of maximal entropy" on $\Omega^{(\alpha)}$. It is well-known that the measure of maximal entropy of a subshift of finite type is a Gibbs measure, hence, we are motivated to generalize the notion of Gibbs measure  to our non-dynamical setting.
    
    Based on this observation, we will define Gibbs-like measure on certain symbolic space like $\Omega^{(\alpha)}$ and  develop the dimension theory of it.   
More precisely, at first, we define the abstract symbolic space $\Omega_\alpha$; next, we define the notion of Gibbs-like measure and introduce a family of nice potentials $\mathcal F_\alpha,$ such that for each $\Phi\in \mathcal F_\alpha$  we can associate a Gibbs-like measure; then we will show that such kind of measure is exact upper and lower dimensional. As an application of this theory, we show that $\mu_\alpha$ is indeed a Gibbs-like measure of the potential $\Phi=\{\phi_n: n\ge 1\}$ defined by 
$$
\phi_n(x):= \log q_n(\alpha),
$$
where $q_n(\alpha)$ is the denominator of the $n$-th approximation  of $\alpha$ (we will show that $q_n(\alpha)\le \#\Omega^{(\alpha)}_n\le 5q_n(\alpha)$).
As a consequence, $\mu_\alpha$ is   exact upper and lower dimensional, so does $\NN_{\alpha,\lambda}$.

To construct the DOS which is not exact-dimensional, we make use of the fact that $\dim_H \NN_{\alpha_1,\lambda}\ne \dim_H \NN_{\alpha_2,\lambda}$ for $\lambda$ large enough (see \cite{Q}). We will construct a frequency $\alpha=[0; a_1,a_2,\cdots]$ such that $a_1a_2\cdots= 1^{t_1}2^{\tau_1}1^{t_2}2^{\tau_2}\cdots$. By choosing $t_i$ and $\tau_i$ carefully, we can show that the local dimension of $\NN_{\alpha,\lambda}$ does not exist $\NN_{\alpha,\lambda}$-a.e., hence $\NN_{\alpha,\lambda}$ is not exact-dimensional.

Finally, we say some words on notations.
By $a_n\sim b_n$, we mean that
there exists $C>1$ such that $C^{-1} b_n\le a_n\le C b_n$ for any $n$. By $a_n\sim_D b_n$, we mean that the constant $C$ only depends on $D$. Given two measures $\mu,\nu$ on a measurable space $(X,\mathscr{F})$, $\mu\asymp \nu$ means that there exists a constant $C>1$ such that $C^{-1}\mu(A)\le \nu(A)\le C\mu(A)$ for any $A\in \mathscr{F}.$

  The rest of the paper is organized as follows. In Sect. \ref{sec-space-potential}, we define the symbolic space and the potentials on it. In Sect. \ref{sec-gibbs-like}, we define Gibbs-like measure and study its exact-dimensional property. The result in this section is of independent interest. In Sect. \ref{sec-structure}, we study the structure of the spectrum of Sturm Hamiltonian. In Sect. \ref{sec-proof-main}, we apply the result in Sect. \ref{sec-gibbs-like} to  prove Theorem \ref{main}.

\section{Symbolic space and potentials}\label{sec-space-potential}

The definitions of this section are motivated both by symbolic dynamical system and by the structure of the spectrum of Sturm Hamiltonian.

\subsection{Symbolic space}   

Fix $M\ge 1$. Define $\mathbb{A}:=\{1,\cdots,M\}^\N.$ Assume $\{\mathcal{A}_1,\cdots, \mathcal{A}_M\}$ is  a pair-wise  disjoint  family of  alphabets.  Write
$
\mathcal{A}_i=\{a_{i1},\cdots, a_{in_i}\}.
$
We further assume that $n_i\ge 2$ for $i=1,\cdots,M.$ Let $\mathcal{A}:=\bigcup_{i=1}^M\mathcal{A}_i$, then $\#\A=n_1+\cdots +n_M.$  For any pairs $(\mathcal{A}_i,\mathcal{A}_j)$,  we associate  an $n_i\times n_j$ incidence matrix $A_{ij}$ with  entries $0$ or $1$. For $a_{ik}\in \A_i$ and $a_{jl}\in\A_j$, we say word $a_{ik}a_{jl}$ {\it admissible} and denote by $a_{ik}\to a_{jl}$ if $A_{ij}(k,l)=1,$ where $B(k,l)$ denote the $(k,l)$-th entry of a matrix $B.$ We always assume the following

{\bf Strongly primitive condition}: there exists $p_0\ge 2$ such that for any $m\ge p_0$ and any  $a_1\cdots a_{m}\in \{1,\cdots,M\}^{m}$, 
\begin{equation}\label{strong-primitive}
A_{a_1a_2}\cdot A_{a_2a_3} \cdots A_{(m-1)m} \text{ is a positive matrix}.
\end{equation}

The following property will be repeated used later. It follows directly from \eqref{strong-primitive}.

\begin{lem}\label{connecting}
Fix $m\ge p_0$ and $a_1\cdots a_m\in \{1,\cdots,M\}^m$. Take any $e\in \A_{a_1}$ and $\hat e\in \A_{a_m}$. Then there exists an admissible word $w\in \prod_{i=1}^m\A_{a_i}$ such that $w_1=e$ and $w_m=\hat e.$
\end{lem}

For each $\alpha=a_1a_2\cdots\in \mathbb{A},$ define the {\it symbolic space $\Omega_\alpha$} as 
\begin{equation*}
\Omega_\alpha:=\{x_1x_2\cdots\in \prod_{i=1}^\infty \mathcal{A}_{a_i}: x_i\to x_{i+1} \text{ for all } i\ge 1 \}.
\end{equation*}
$\alpha$ is called the {\it index} of $\Omega_\alpha.$ From the definition, $\Omega_\alpha$ is completed determined by the alphabet sequence $(\A_{a_1}, \A_{a_2},\cdots )$ and incidence matrix sequence $(A_{a_1a_2}, A_{a_2a_3}, \cdots)$. Given $x\in \Omega_\alpha$ and $n\in \N$, write $x|_n:=x_1\cdots x_n.$ 

Define a metric on  $ \prod_{i=1}^\infty \mathcal{A}_{a_i}$ as 
\begin{equation*}
\hat d(x,y): =2^{-|x\wedge y|},
\end{equation*}
 where $x\wedge y$ denotes the maximal common prefix of $x$ and $y$, then $ \prod_{i=1}^\infty \mathcal{A}_{a_i}$ becomes a compact metric space. It is seen that $\Omega_\alpha$ is a closed subset of $ \prod_{i=1}^\infty \mathcal{A}_{a_i}$, hence also compact.

\begin{rem}\label{rem2.4}
{\rm 
If $\alpha=\kappa^\infty$, then 
 there is only one alphabet $\mathcal{A}_\kappa$, and  only one  incidence matrix $A_{\kappa\kappa}$, which is primitive.   In this case,     $ \Omega_{\alpha}$ is a subshift of finite type with alphabet $\A_\kappa$ and incidence matrix $A_{\kappa\kappa}$. 
}
\end{rem}

Fix $\vec{a}=a_1\cdots a_n\in \{1,\cdots,M\}^n$, 
define 
$$
\Omega_{\vec{a},n}:=\{w_1\cdots w_n\in\prod_{i=1}^n \mathcal{A}_{a_i} : w_i\to w_{i+1}  \}.
$$
For $\alpha=a_1a_2\cdots\in \mathbb{A}$, write $\alpha|_n:=a_1\cdots a_n$, 
\begin{equation*}\label{Omega-alpha-n}
\Omega_{\alpha,n}:=\Omega_{\alpha|_n,n}\ \ \ \text{  and }\ \ \ \Omega_{\alpha,\ast}:=\bigcup_{n\ge 1} \Omega_{\alpha,n}.
\end{equation*}

Any $w\in \Omega_{\alpha,\ast}$ is called an {\it admissible} word. For $w\in \Omega_{\alpha,\ast}$, $|w|$ denotes the length of  $w$. 
If $v,w\in \Omega_{\alpha,\ast}$ and $w=vu$,  $v$ is called  a {\it prefix} of $w$ and denoted by $v\prec w.$
 
 For any $w\in \Omega_{\alpha,n}$ and $m\ge 0,$ define the set of  {\it descendents  of $w$ of $m$-th generation} as 
\begin{equation}\label{descendent}
 \Xi_{w,m}(\alpha):=\{v\in \Omega_{\alpha,n+m}: w\prec v\}.
\end{equation}
 By \eqref{strong-primitive} and the definition of admissibility, it is seen that 
 \begin{equation}\label{num-xi}
 1\le \#\Xi_{w,m}(\alpha)\le (\#\A)^{m}. 
 \end{equation}

 Given $w\in \Omega_{\alpha,n}$, the {\it cylinder} $[w]_\alpha$ in $\Omega_\alpha$ is defined as 
$$
[w]_\alpha:=\{x\in \Omega_\alpha: x|_n=w\}. 
$$
It is well-known that $[w]_\alpha$ is both open and closed, and $\{[w]_\alpha: w\in \Omega_{\alpha,\ast}\}$ forms a basis of the topology on $\Omega_\alpha.$

We denote by $\mathcal{M}(\Omega_\alpha)$ the set of probability measures supported on $\Omega_\alpha.$

\subsection{potential and weak Gibbs metric on $\Omega_\alpha$}

To introduce the Gibbs-like measure  and the weak Gibbs metric on $\Omega_\alpha$, we need to consider  potentials   defined on $\Omega_\alpha$.   
\subsubsection{Potentials on $\Omega_\alpha$}

Given a family of continuous functions $\phi_n: \Omega_\alpha\to \mathbb{R}$. Write $\Phi=\{\phi_n:n\ge1\}.$   $\Phi$ is called a {\it potential}.  

We call $\Phi$ {\it regular} if there exists a constant $C_{\rm rg}(\Phi)>0$ such that for any $n,m\in\N$, 
\begin{equation}\label{regular}
\|\phi_n-\phi_m\|_\infty\le C_{\rm rg}(\Phi)|n-m|.
\end{equation}

   We say that $\Phi$ has  {\it bounded variation} if there exists a constant $C_{\rm bv}(\Phi)\ge 0$ such that for any $n\in \N,$
\begin{equation}\label{b-dis}
\sup\{|\phi_n(x)-\phi_n(y)|: x,y\in \Omega_\alpha, x|_n=y|_n\}\le C_{\rm bv}(\Phi).
\end{equation}

we say that $\Phi$ has {\it bounded covariation } if there exists a constant $C_{\rm bc}(\Phi)\ge0$ such that  for any $w=uv,\tilde w=\tilde u v\in \Omega_{\alpha,\ast}$ and any $x\in [w]_\alpha, \tilde x\in [\tilde w]_\alpha$, 
\begin{equation}\label{b-cov}
|(\phi_{|w|}(x)-\phi_{|u|}(x))-( \phi_{|\tilde w|}(\tilde x)-\phi_{|\tilde u|}(\tilde x))|\le C_{\rm bc}(\Phi).
\end{equation}

We call $\Phi$ {\it positive}  if  $\phi_n(x)\uparrow +\infty$  for any $x\in \Omega_\alpha$. We call $\Phi$ {\it negative}, if $-\Phi$ is positive. 

Denote by $\widehat{\mathcal{F}}_\alpha $  the set of all regular potentials with bounded variation. Denote by $  {\mathcal{F}}_{\alpha}$ the set of  all potentials in $\widehat{\mathcal{F}}_\alpha$  with bounded covariation. Define 
$$
\widehat{\mathcal{F}}_\alpha^+:=\{\Phi\in \widehat{\mathcal{F}}_\alpha: \Phi \text{ is positive}\}\ \ \text{ and }\ \ {\mathcal{F}}_\alpha^+:=\{\Phi\in {\mathcal{F}}_\alpha: \Phi \text{ is positive}\}.
$$
$\widehat{\mathcal{F}}_\alpha^-$ and ${\mathcal{F}}_\alpha^-$ can be defined similarly.

\subsubsection{Weak Gibbs metric  on $\Omega_\alpha$}

The following definition is motivated by  \cite{GP97,KS04,BQ}. 
Fix a  potential $\Psi\in \widehat{\mathcal{F}}_\alpha^-$. For any $w\in \Omega_{\alpha,\ast}$, define 
\begin{equation*}
r_w(\Psi):=\sup_{x\in[w]_\alpha} \exp(\psi_{|w|}(x)).
\end{equation*}
For any  $x,y\in \Omega_\alpha$, define    
\begin{equation}\label{weak-gibbs}
d_\Psi(x,y):=
\begin{cases}
r_{x\wedge y}(\Psi)& x\ne y,\\
0& x=y.
\end{cases}
\end{equation}

\begin{lem}\label{lem-weak-gibbs}
 (i) $d_\Psi$ is a ultrametric on $\Omega_\alpha$ and it induces the same topology on $\Omega_\alpha$ as $\hat d$ does.    $(\Omega_\alpha,d_\Psi)$ has no isolated point. 
 
 (ii) Define $c:=\exp(-C_{\rm bv}(\Psi)-p_0 C_{\rm rg}(\Psi)), $ then for any $w\in \Omega_{\alpha,n}$, 
  \begin{equation}\label{diameter}
 c\cdot r_w(\Psi)\le {\rm diam}([w]_\alpha)\le r_w(\Psi),
  \end{equation}
 Define $\hat c:=\exp(-C_{\rm bv}(\Psi))$, then for any $x\in [w]_\alpha$, 
\begin{equation}\label{in-out-ball}
B(x, \hat c\cdot r_w(\Psi))\subset [w]_\alpha\subset \overline{B}(x,r_w(\Psi)),
\end{equation}
where $\overline{B}$ means closed ball.
  
 (iii) $(\Omega_\alpha, d_\Psi)$ has Besicovitch's covering property.
\end{lem}

\begin{proof}
(i)\ 
It is obvious that $d_\Psi(x,y)\ge 0$, $d_\Psi(x,y)=0$ if and only if $x=y$ and $d_\Psi(x,y)=d_\Psi(y,x)$. Let us check the triangle inequality. 

{\bf Claim 1:} If $x\wedge z \prec x\wedge y$, then $d_\Psi(x,y)\le d_\Psi(x,z)$.

$\lhd $  Write  $m:=|x\wedge z|$ and $n:=|x\wedge y|$, then $m\le n.$ Assume $t\in [x\wedge y]_\alpha$ is such that 
$d_\Psi(x,y)=\exp(\psi_n(t))$, we have  $\psi_n(t)\le \psi_m(t)$ since $\Psi$ is negetive. Since $t\in[x\wedge y]_\alpha\subset [x\wedge z]_\alpha$, we conclude that $d_\Psi(x,y)\le d_\Psi(x,z).$
\hfill $\rhd$

 Given $x,y\in \Omega_\alpha,$ without loss of generality, we assume $x\ne y.$  Assume $|x\wedge y|=n$, then  $x_{n+1}\ne y_{n+1}$. Take any $z\in \Omega_\alpha.$ If $z|_n\ne x\wedge y,$ then  $x\wedge z \prec x\wedge y$. By Claim 1, $d_\Psi(x,y)\le d_\Psi(x,z).$ If  $z|_n=x\wedge y$,
 then at least one of $x\wedge z$ and $y\wedge z$ is equal to $x\wedge y$ since $x_{n+1}\ne y_{n+1}$. Hence $d_\Psi(x,y)\le \max\{d_\Psi(x,z), d_\Psi(y,z)\}.$
Thus $d_\Psi$ is a ultrametric on $\Omega_\alpha.$

Write $\psi_n^\ast=\max\{\psi_n(x): x\in \Omega_\alpha\}$ and $\psi_{n\ast}=\min\{\psi_n(x): x\in \Omega_\alpha\}$. Since $\Psi$ is negative, we conclude that $\psi_n^\ast\downarrow -\infty$ as $n\to\infty.$ By the definition, we have 
$$
\exp(\psi_{|x\wedge y|\ast})\le d_\Psi(x,y)\le \exp(\psi_{|x\wedge y|}^\ast)\ \ \text{ and }\ \ \hat d(x,y)=2^{-|x\wedge y|}.
$$
From this, it is easy to show that $d_\Psi$ and $\hat d$ induce the same topology.

Fix any $x\in \Omega_\alpha,$ by Lemma \ref{connecting}, we can construct a sequence $\{y^k:k\ge1\}$ such that $x\ne y^k$ and  $|x\wedge y^k|\uparrow \infty,$  thus $d_\Psi(x,y^k)\to 0$. This means that $x$ is not isolated.

(ii)\ If $x,y\in [w]_\alpha,$ then $[x\wedge y]_\alpha\subset[w]_\alpha$.  Since $\Psi$ is negetive,
\begin{equation}\label{include}
d_\Psi(x,y)=\sup_{z\in[x\wedge y]_\alpha} \exp(\psi_{|x\wedge y|}(z))\le  \sup_{z\in[w]_\alpha} \exp(\psi_{|w|}(z))=r_w(\Psi).
\end{equation}
On the other hand,  by Lemma \ref{connecting},  there exist $x,y\in [w]_\alpha $ such that $x\wedge y=wu$ with $m:=|u|\le p_0.$ Assume $t\in [x\wedge y]_\alpha$ is such that $\exp(\psi_{|w|+m}(t))=d_\Psi(x,y)$ and  $\tilde t\in [w]_\alpha$ is such that $\exp(\psi_{|w|}(\tilde t))=\sup_{z\in[w]_\alpha} \exp(\psi_{|w|}(z)).$ Then by the regular property and bounded variation of $\Psi$, 
\begin{eqnarray*}
d_\Psi(x,y)&=&\exp(\psi_{|w|+m}(t))=\exp(\psi_{|w|}(t))\exp(\psi_{|w|+m}(t)-\psi_{|w|}(t))\\
&\ge&\exp(\psi_{|w|}(t))\exp(-m C_{\rm rg}(\Psi))\ge \exp(\psi_{|w|}(t))\exp(-p_0 C_{\rm rg}(\Psi))\\
&\ge&\exp(\psi_{|w|}(\tilde t))\exp(\psi_{|w|}(t)-\psi_{|w|}(\tilde t))\exp(-p_0 C_{\rm rg}(\Psi))\\
&\ge&\exp(\psi_{|w|}(\tilde t))\exp(-C_{\rm bv}(\Psi)-p_0 C_{\rm rg}(\Psi)).
\end{eqnarray*}
Then \eqref{diameter} holds.

Now fix $x\in [w]_\alpha$, by \eqref{include}, $[w]_\alpha\subset \overline{B}(x,r_w(\Psi)).$ If $y\not\in [w]_\alpha$, then $x\wedge y$ is a strict prefix of $w$ and hence $m:=|x\wedge y|<|w|.$ Assume $\tau\in[x\wedge y]_\alpha$ is such that $d_\Psi(x,y)=\exp(\psi_m(\tau))$ and $\tilde \tau\in[w]_\alpha$ is such that $r_w(\Psi)=\exp(\psi_{|w|}(\tilde \tau))$.  Since $\Psi$ is negative and has bounded variation,  we have 
\begin{eqnarray*}
d_\Psi(x,y)&=&\exp(\psi_m(\tau))=\exp(\psi_m(\tilde \tau))\exp(\psi_m(\tau)-\psi_m(\tilde \tau))\\
&\ge&\exp(\psi_m(\tilde \tau))\exp(-C_{\rm bv}(\Psi))\\
&=&\exp(\psi_{|w|}(\tilde \tau))\exp(\psi_m(\tilde \tau)-\psi_{|w|}(\tilde \tau))\exp(-C_{\rm bv}(\Psi))\\
&\ge&\exp(-C_{\rm bv}(\Psi))r_w(\Psi).
\end{eqnarray*}
This means that $B(x, \hat c\cdot r_w(\Psi))\subset [w]_\alpha.$ Thus \eqref{in-out-ball} holds.

(iii)\ At first we show the following:

{\bf Claim 2:} For any $r>0$ and $x\in \Omega_\alpha$, the  closed ball $\overline{B}(x,r)$ is a cylinder.

$\lhd$
 Since $x$ is not isolated, there exists $y\in \overline{B}(x,r)\setminus\{x\}.$ Let us show that $[x\wedge y]_\alpha\subset \overline{B}(x,r).$ Indeed, if $z\in [x\wedge y]_\alpha$, then $x\wedge y\prec x\wedge z$. By Claim 1, $d_\Psi(x,z)\le d_\Psi(x,y)\le r,$ so $z\in \overline{B}(x,r)$.  
  
  Define $n_0:=\inf\{|x\wedge y|: y\in \overline{B}(x,r)\}$.
Let us show that  $\overline{B}(x,r)=[x|_{n_0}]_\alpha.$

 Since $\overline{B}(x,r)\setminus \{x\}$ is nonempty, there exists $y\in \overline{B}(x,r)$ such that $x\wedge y=x|_{n_0}.$ We already showed that $[x|_{n_0}]_\alpha=[x\wedge y]_\alpha\subset \overline{B}(x,r)$.
If there exists $z\in \overline{B}(x,r)\setminus [x|_{n_0}]_\alpha,$ then $x\wedge z$ must be a strict prefix of $x|_{n_0},$ hence $|x\wedge z|<n_0$, which contradicts with the definition of $n_0$.
\hfill $\rhd$

 Now assume $A\subset \Omega_\alpha$ and $\mathcal B$ is a family of closed balls such that  for any  $x\in A$, there exists a closed ball $\overline{B}(x,r_x)\in \mathcal B.$ By Claim 2, $\overline{B}(x,r_x)=[w^x]_\alpha$ for some $w^x\in \Omega_{\alpha,\ast}$.  
 Write  $\mathcal U_0:=\{w^x: x\in A\}$. We define two sequences $\mathcal W_n$ and $\mathcal U_n$ inductively as follows. Assume $\mathcal U_{n-1}$ has been defined, write  $m_{n-1}:=\min\{|w|: w\in \mathcal U_{n-1}\}$. Define 
 $$
 \mathcal W_n:= \Omega_{\alpha,m_{n-1}}\cap \mathcal U_{n-1}.
 $$
 Define $\mathcal U_n$ as 
 $$
 \mathcal U_n:=\{w\in \mathcal U_{n-1}: u\not \prec w \text{ for any } u\in \mathcal W_n\}.
 $$
 If $\mathcal U_n= \emptyset$, then the process stops. Otherwise, 
 $m_n>m_{n-1}$, thus induction can continue. By induction, we finish the definition of $\mathcal W_n$ and $\mathcal U_n$. 
 
  Now define $\mathcal W_\infty:=\bigcup_{n\ge 1}\mathcal W_n.$
 
 {\bf Claim 3: } If $w, \tilde w\in \mathcal W_\infty$ and $w\ne \tilde w$, then $[w]_\alpha\cap [\tilde w]_\alpha=\emptyset.$ Moreover 
 $$A \subset  \bigcup_{w\in \mathcal W_\infty} [w]_\alpha. $$
 
 $\lhd$ By the way we define $\mathcal W_n$, if $w\ne \tilde w,$ then $w\not\prec \tilde w$ and $\tilde w\not\prec w,$ hence $[w]_\alpha\cap [\tilde w]_\alpha=\emptyset.$ Given  $x\in A$. Either $w^x\in \mathcal W_\infty,$ then $x\in \overline{B}(x,r_x)=[w^x]_\alpha.$ Or $w^x\not\in \mathcal W_\infty,$ thus $w^x$ is gotten rid of from $\mathcal U_{k-1}$ for some $k.$ However, this means that there exists  some $w\in \mathcal W_k$ such that $w\prec w^x.$ Consequently, $x\in \overline{B}(x,r_x)=[w^x]_\alpha\subset [w]_\alpha.$
  \hfill $\rhd$
 
 Claim 3 obviously implies that $(\Omega_\alpha,d_\Psi)$ has Besicovitch's covering property with multiplicity $1$ (see for example \cite{M} Chapter 2 for the related definition).
\end{proof}

\section{Exact-dimensional property of Gibbs-like measure}\label{sec-gibbs-like}

At first, we define Gibbs-like measure and give criterion for its existence.  Then we fix a weak Gibbs metric on $\Omega_\alpha$  and  study the exact-dimensional properties  of Gibbs-like measures. 
We write $\alpha$ as $\alpha=a_1a_2a_3\cdots.$

\subsection{Definition and existence of Gibbs-like measure}

Our definition of Gibbs-like measure is inspired   by \cite{MRW,FLW}, where the measure is defined in a geometric way. 

Given $\Phi\in\widehat{\mathcal{F}}_\alpha$,    $\mu\in \mathcal{M}(\Omega_\alpha)$  is called  a {\it Gibbs-like} measure of $\Phi$, if there exists a constant $C>1$ such that  for any $x\in\Omega_\alpha $,
\begin{equation}\label{def-gibbs-like}
\frac{C^{-1}\exp(\phi_n(x))}{\sum_{w\in\Omega_{\alpha,n}} \exp(\phi_n(x_w))}\le \mu([x|_n]_\alpha)\le \frac{C\exp(\phi_n(x))}{\sum_{w\in\Omega_{\alpha,n}} \exp(\phi_n(x_w))},
\end{equation}
where for any $w\in\Omega_{\alpha,n},$ $x_w$ is any fixed point in $ [w]_\alpha$.

We have the following criterion for the existence of Gibbs-like  measure:

\begin{thm}\label{exist-gibbs-like}
If $\Phi\in {\mathcal{F}}_\alpha$, then $\Phi$ has a Gibbs-like measure. Moreover the constant $C$ in \eqref{def-gibbs-like} only depends on $C_{\rm rg}(\Phi), C_{\rm bv}(\Phi), C_{\rm bc}(\Phi), \#\A$ and $p_0$.
\end{thm}

The proof relies on a technical lemma which we will state now. 

From now on, 
for any $w\in \Omega_{\alpha,\ast}$, we fix  $x_w\in [w]_\alpha$ once for all. Given a $\Phi\in \F_\alpha$ and  $w\in \Omega_{\alpha,n}$, define 
\begin{equation*}\label{def-sigma-n}
\begin{cases}
\sigma_n=\sigma_n(\Phi):=&\sum_{w\in\Omega_{\alpha,n}} \exp(\phi_n(x_w)),\\
\sigma^w_{m}=\sigma^w_{m}(\Phi):=&\sum_{v\in \Xi_{w,m}(\alpha)} \exp(\phi_{n+m}(x_v)-\phi_n(x_w)),
\end{cases}
\end{equation*}
 where $\Xi_{w,m}(\alpha)$ is defined by \eqref{descendent}.

\begin{lem}
Fix $\Phi\in \F_\alpha$ and define $\sigma_n, \sigma^w_{m}$ as above. 
Then there exists a constant $C>1$ depending only on $C_{\rm rg}(\Phi), C_{\rm bv}(\Phi), C_{\rm bc}(\Phi), \#\A$ and $p_0$  such that for any $n\in \N$, $m\ge 0$ and any $w,\tilde w\in \Omega_{\alpha,n}$,
\begin{equation}\label{s-m-w}
C^{-1} \sigma^{\tilde w}_m\le \sigma^w_{m}\le C \sigma^{\tilde w}_m.
\end{equation}
Consequently for any $w\in \Omega_{\alpha,n},$
\begin{equation}\label{sigma-n-m}
C^{-1}\sigma_n \sigma^w_{m}\le \sigma_{n+m}\le C\sigma_n \sigma^w_{m}.
\end{equation}
\end{lem}

\proof\ Let $C_1, C_2, C_3$ be the constants in \eqref{regular}, \eqref{b-dis} and \eqref{b-cov}, respectively.  Recall that  $p_0$ is  such that \eqref{strong-primitive} holds.

At first we assume $0\le m\le p_0.$ If $w\prec v,$ by regularity and bounded variation of $\Phi,$ we have $|\phi_{n+m}(x_v)-\phi_n(x_w)|\le p_0C_1+C_2$.  Then by \eqref{num-xi}, 
\begin{equation}\label{sigma-w-m}
\exp(-p_0C_1-C_2)\le \sigma^w_{m}\le (\#\A)^{p_0}\exp(p_0C_1+C_2).
\end{equation}
Thus \eqref{s-m-w} holds for $\tilde C=(\#\A)^{p_0}\exp(2p_0C_1+2C_2).$

Next we assume $m> p_0$.  Fix any $e\in \mathcal{A}_{a_{n+p_0}}$ and define
$$
\begin{cases}
X_e=\{u\in \prod_{j=n+1}^{n+p_0-1}\mathcal{A}_{a_j}: wue \text{ admissible }\}\\
\tilde X_e=\{\tilde u\in \prod_{j=n+1}^{n+p_0-1}\mathcal{A}_{a_j}: \tilde w\tilde ue \text{ admissible }\}\\
Y_e=\{v\in \prod_{j=n+p_0+1}^{n+m}\mathcal{A}_{a_j}: ev\text{ admissible }\}.
\end{cases}
$$
Then  we have 
\begin{equation}\label{s-w-s-tilde-w}
\begin{cases}
\sigma^w_{m}&=\sum_{e\in \mathcal{A}_{a_{n+p_0}}}\sum_{v\in Y_e}\sum_{u\in X_e} \exp(\phi_{n+m}(x_{wuev})-\phi_n(x_w))\\
\sigma^{\tilde w}_m&=\sum_{e\in \mathcal{A}_{a_{n+p_0}}}\sum_{v\in Y_e}\sum_{\tilde u\in \tilde X_e} \exp(\phi_{n+m}(x_{\tilde w\tilde uev})-\phi_n(x_{\tilde w})).
\end{cases}
 \end{equation}
By Lemma \ref{connecting} and the definition of admissibility,   we have 
\begin{equation}\label{ulbd}
1\le \#X_e,\#\tilde X_e\le (\#\A)^{p_0}.
\end{equation}
 
For any  $u\in X_e$ and $v\in Y_e$ we have
\begin{eqnarray*}
&&\phi_{n+m}(x_{wuev})-\phi_n(x_w)\\
&=&\phi_{n+m}(x_{wuev})-\phi_n(x_{wuev})+\phi_n(x_{wuev})-\phi_n(x_w)\\
&\le &\phi_{n+m}(x_{wuev})-\phi_n(x_{wuev})+C_2\\
&=&\phi_{n+m}(x_{wuev})-\phi_{n+p_0}(x_{wuev})+\phi_{n+p_0}(x_{wuev})-\phi_n(x_{wuev})+C_2\\
&\le &\phi_{n+m}(x_{wuev})-\phi_{n+p_0}(x_{wuev})+p_0C_1+C_2,
\end{eqnarray*}
where  the second equation is due to bounded variation of $\Phi,$ the fourth equation is due to the fact that $\Phi$ is regular. 
By the same proof we get 
$$
\phi_{n+m}(x_{wuev})-\phi_n(x_w)\ge \phi_{n+m}(x_{wuev})-\phi_{n+p_0}(x_{wuev})-p_0C_1-C_2.
$$
 Similarly, for any  $\tilde u\in \tilde X_e$ and $v\in Y_e$ we have
 $$
 |\phi_{n+m}(x_{\tilde w\tilde uev})-\phi_n(x_{\tilde w})-\left(\phi_{n+m}(x_{\tilde w\tilde uev})-\phi_{n+p_0}(x_{\tilde w\tilde uev})\right)|\le p_0C_1+C_2.
 $$
 By  bounded covariation of $\Phi$, we have 
 \begin{equation}\label{simm}
 |\phi_{n+m}(x_{wuev})-\phi_n(x_w)-\left(\phi_{n+m}(x_{\tilde w\tilde uev})-\phi_n(x_{\tilde w})\right)|\le 2p_0C_1+2C_2+C_3.
 \end{equation}
 Define $C=(\#\A)^{p_0}\exp(2p_0C_1+2C_2+C_3),$ we get \eqref{s-m-w} by
combining \eqref{s-w-s-tilde-w}, \eqref{ulbd} and \eqref{simm}.

Now  for any fixed $w\in\Omega_{\alpha,n}$, by \eqref{s-m-w} we have 
\begin{eqnarray*}
\sigma_{n+m}&=&\sum_{u\in \Omega_{\alpha,n+m}} \exp(\phi_{n+m}(x_u))
=\sum_{\tilde w\in \Omega_{\alpha,n}}\sum_{u\in \Xi_{\tilde w,m}(\alpha)} \exp(\phi_{n+m}(x_u))\\
&=&\sum_{\tilde w\in \Omega_{\alpha,n}} \sigma^{\tilde w}_m \exp(\phi_n(x_{\tilde w}))
\le C\ \sigma^{w}_m\sum_{\tilde w\in \Omega_{\alpha,n}}  \exp(\phi_n(x_{\tilde w}))
=C \sigma^{w}_m\sigma_n.
\end{eqnarray*}
$\sigma_{n+m}\ge C^{-1}\sigma^{w}_m\sigma_n$ follows in the same way. 
\hfill $\Box$

\noindent{\bf Proof of Theorem \ref{exist-gibbs-like}.}\  For any $n\in \N$, we define a measure $\mu_n$ on the Borel $\sigma$-algebra generalized by the $n$-th cylinders as follows: for any $w\in \Omega_{\alpha,n},$ let 
$
\mu_n([w]_\alpha):={\exp(\phi_n(x_w))}/{\sigma_n}.
$
Assume $\mu$ is any weak-star limit of $\{\mu_n:n\ge1\}$.  Let us show that $\mu$ is a  Gibbs-like measure of $\Phi$.

Fix any $w\in \Omega_{\alpha,n}$ and $m\ge0$. By \eqref{sigma-n-m}, we have 
\begin{eqnarray*}
\mu_{n+m}([w]_\alpha)&=&\sum_{u\in\Xi_{w,m}(\alpha)}\mu_{n+m}([u]_\alpha)
=\sigma_{n+m}^{-1}\sum_{u\in \Xi_{w,m}(\alpha)} \exp(\phi_{n+m}(x_u))\\
&=&\sigma_{n+m}^{-1}\exp(\phi_n(x_w))\sum_{u\in \Xi_{w,m}(\alpha)} \exp(\phi_{n+m}(x_u)-\phi_n(x_w)) \\
&=&\sigma_{n+m}^{-1}\exp(\phi_n(x_w))\cdot \sigma^w_m
\sim_C \ \frac{\exp(\phi_n(x_w))}{\sigma_n}.
\end{eqnarray*}
 Notice that $[w]_\alpha$ is open and closed, since some subsequence  $\mu_{n+m_l}$ converge weakly to $\mu$, let $l\to\infty$ we get 
 $$
 \mu([w]_\alpha)\sim_C\  \frac{\exp(\phi_n(x_w))}{\sigma_n}.
 $$
Since $\Phi$ has bounded variation, the above property implies that $\mu$ is a Gibbs-like measure, and  the constant in \eqref{def-gibbs-like} only depends on $C_{\rm rg}(\Phi),$ $ C_{\rm bv}(\Phi), C_{\rm bc}(\Phi), \#\A$ and $p_0$.
\hfill $\Box$

\subsection{Exact-dimensional properties of Gibbs-like measure}

Fix  $\Psi\in {\mathcal{F}}_\alpha^-$ and define the weak Gibbs metric $d_\Psi$ on $\Omega_\alpha$ by \eqref{weak-gibbs}.    Gibbs-like measure behaves  well on the metric space $(\Omega_\alpha,d_\Psi)$  in the following sense:

\begin{thm}\label{exact-hp-dim}
 Fix  $\Phi\in{\mathcal{F}}_\alpha$ and let $\mu$ be a Gibbs-like measure of $\Phi,$ then $\mu$ is   exact upper and lower dimensional.
\end{thm}

\proof\
At first we show that $\mu$ is   exact lower dimensional. We show it by contradiction. If $\mu$ is not  exact lower dimensional, then there exist two disjoint  Borel subsets $X,Y\subset \Omega_\alpha$ with $\mu(X), \mu(Y)>0$ and two numbers $d_1<d_2$ such that 
$$
 \underline{d}_{\mu}(x)\le d_1 \ (\forall x\in X)\ \ \ \text{ and }\ \ \ \underline{d}_{\mu}(y)\ge d_2 \ (\forall y\in Y).
$$
By Lemma  \ref{lem-weak-gibbs} (iii), the space $\Omega_\alpha$ has Besicovitch's covering property.
By the density property for Radon measure (see for example \cite{M} Chapter 2), there exist $x\in X$ and $y\in Y$ such that 
$$
\lim_{n\to\infty}\frac{\mu(X\cap [{x|_n}]_\alpha)}{\mu([{x|_n}]_\alpha)}=1\ \ \ \text{ and }\ \ \ \lim_{n\to\infty}\frac{\mu(Y\cap [y|_n]_\alpha)}{\mu([y|_n]_\alpha)}=1.
$$
Let $ C_1,C_2,C_3, C_4, C_5>0$ be the constants in \eqref{regular}, \eqref{b-dis}, \eqref{b-cov}, \eqref{def-gibbs-like} and  \eqref{sigma-n-m},   respectively.
Recall that  $p_0$ is such that \eqref{strong-primitive} holds.
By \eqref{sigma-w-m}, there exists a constant $C_6>0$ such that  $\sigma^w_{p_0}\le C_6$ for any $w\in \Omega_{\alpha,\ast}$. 

Fix $\delta>0$ such that  
\begin{equation}\label{def-delta}
\delta<\min\{\left((4C_4^4\exp(2C_2+C_3)+1)C_4^2C_{5}C_6\exp(p_0C_1+C_2)\right)^{-1},1/2\}.
\end{equation}
Assume $q$ is such that 
\begin{equation}\label{density}
\begin{cases}
\mu(X\cap [x|_{q}]_\alpha)\ge (1-\delta)\mu([x|_q]_\alpha), \\  
\mu(Y\cap [y|_q]_\alpha)\ge (1-\delta)\mu([y|_q]_\alpha).
\end{cases}
\end{equation}

\medskip

\noindent {\bf Claim 1:}  for any $v\in \Xi_{x|_q, p_0}(\alpha)$ and $w\in \Xi_{y|_q, p_0}(\alpha)$, 
\begin{equation}\label{density1}
\begin{cases}
\mu(X\cap [v]_\alpha)\ge (1-C\delta)\mu([v]_\alpha), \\  
\mu(Y\cap [w]_\alpha)\ge (1-C\delta)\mu([w]_\alpha),
\end{cases}
\end{equation}
where $C=C_4^2C_{5}C_6\exp(p_0C_1+C_2).$

\noindent $\lhd$ We only show the first inequality  of \eqref{density1}, since the proof of the second one is the same.
We have $[x|_q]_\alpha=\bigcup_{v\in \Xi_{x|_q, p_0}(\alpha)}[v]_\alpha$.  Define 
$$
\eta:=\max\left\{\frac{\mu([v]_\alpha\setminus X)}{\mu([v]_\alpha)}:v\in \Xi_{x|_q, p_0}(\alpha)\right\}.
$$
If $\eta=0,$  
the result holds trivially.  So we assume $\eta>0$ and  $\hat v\in \Xi_{x|_q, p_0}(\alpha)$ attains the maximum. 
Consequently by \eqref{density},
$$
\eta\mu([\hat v]_\alpha)= \mu([\hat v]_\alpha\setminus X)\le \mu([x|_q]_\alpha\setminus X)\le \delta \mu([x|_q]_\alpha).
$$
By Gibbs-like property of $\mu$,  we have 
\begin{eqnarray*}
\mu([\hat v]_\alpha)\ge C_4^{-1}\frac{\exp(\phi_{q+p_0}(x_{\hat v}))}{\sigma_{q+p_0}}\ \ \text{ and }\ \ \mu([x|_q]_\alpha)\le C_4\frac{\exp(\phi_{q}(x))}{\sigma_{q}}. 
\end{eqnarray*}
By \eqref{sigma-n-m} and \eqref{sigma-w-m}, we have   $ \sigma_{q+p_0}\le C_{5}\sigma_q\sigma^{x|_q}_{p_0}\le C_5C_6\sigma_q$. By  regularity and bounded variation of $\Phi,$ we have   
$$
|\phi_q(x)-\phi_{q+p_0}(x_{\hat v})|\le|\phi_q(x)-\phi_{q}(x_{\hat v})|+|\phi_q(x_{\hat v})-\phi_{q+p_0}(x_{\hat v})| \le C_2+ p_0C_1.
$$
Thus 
\begin{equation*}
\eta\le \delta \frac{\mu([x|_q]_\alpha)}{\mu([\hat v]_\alpha)}
\le C_4^2\frac{\sigma_{q+p_0}}{\sigma_q}\exp(\phi_q(x)-\phi_{q+p_0}(x_{\hat v}))\delta
\le C_4^2C_{5}C_6\exp(p_0C_1+C_2)\delta.
\end{equation*}
Then the result follows. 
 \hfill $\rhd$

By Lemma \ref{connecting}, we can find $v\in \Xi_{x|_q, p_0}(\alpha)$ and $w\in \Xi_{y|_q, p_0}(\alpha)$ such that  $v_{q+p_0}=w_{q+p_0}$. We fix such a pair $(v,w).$

\medskip

\noindent{\bf Claim 2:} There exist $\tilde x\in [v]_\alpha\cap X$ and $\tilde y\in [w]_\alpha\cap Y$ such that $\tilde x$ and $\tilde y$ have the same tail, i.e., there exist $l\ge q+p_0$ and $z\in \prod_{j=l+1}^\infty \A_{a_j}$ such that 
$\tilde x=\tilde x|_l\cdot  z$ and $\tilde y=\tilde y|_l\cdot z.$ 

\noindent$\lhd$ We show it by contradiction.  Assume for any  $\tilde x\in [v]_\alpha\cap X $
and any $\tilde y\in [w]_\alpha\cap Y$, $\tilde x$ and $\tilde y$ have no the same tail.

 By \eqref{density1} we have $\mu(X\cap [v]_\alpha)>(1-C\delta)\mu([v]_\alpha).$ Take a compact set $\hat X\subset   X\cap [v]_\alpha$ such that $\mu(\hat X)>(1-2C\delta)\mu([v]_\alpha).$ Consequently $\mu([v]_\alpha\setminus\hat X)\le 2C\delta\mu([v]_\alpha)$. Notice that $[v]_\alpha\setminus \hat X$ is an open set, thus it is a countable disjoint union of cylinders:
$
[v]_\alpha\setminus \hat X=\bigcup_{j\ge1} [{v w_j}]_\alpha,
$
where  different $w_j$ are non compatible. 
Thus we get 
\begin{equation}\label{2delta}
\sum_{j\ge 1} \frac{\mu([{v w_j}]_\alpha)}{\mu([v]_\alpha)}\le 2C\delta. 
\end{equation}

Since $v_{q+p_0}=w_{q+p_0}$  and  for any  $\tilde x\in [v]_\alpha\cap X $
and any $\tilde y\in [w]_\alpha\cap Y$, $\tilde x$ and $\tilde y$ have no the same tail, we must have 
\begin{equation}\label{inclusion}
[w]_\alpha\cap Y\subset \bigcup_{j\ge1} [{w w_j}]_\alpha.
\end{equation}
By the Gibbs-like property of $\mu,$ we have 
$$
\begin{cases}
\frac{\mu([v w_j]_\alpha)}{\mu([v]_\alpha)}&\ge C_4^{-2}\cdot\frac{\exp(\phi_{q+p_0+|w_j|}(x_{vw_j}))}{\exp(\phi_{q+p_0}(x_{v}))}\cdot\frac{\sigma_{q+p_0}}{\sigma_{q+p_0+|w_j|}}\\
\frac{\mu([w w_j]_\alpha)}{\mu([w]_\alpha)}&\le C_4^2\cdot\frac{\exp(\phi_{q+p_0+|w_j|}(x_{w w_j}))}{\exp(\phi_{q+p_0}(x_{w}))}\cdot\frac{\sigma_{q+p_0}}{\sigma_{q+p_0+|w_j|}}.
\end{cases}
$$
By  bounded variation and covariation of $\Phi$, 
\begin{eqnarray*}
&&\frac{\mu([w w_j]_\alpha)}{\mu([w]_\alpha)}/ \frac{\mu([v w_j]_\alpha)}{\mu([v]_\alpha)}\\
&\le&C_4^4\exp\Big(\phi_{q+p_0+|w_j|}(x_{w w_j})-\phi_{q+p_0}(x_{w})\\
&&-(\phi_{q+p_0+|w_j|}(x_{vw_j})-\phi_{q+p_0}(x_{v}))\Big)\\
&\le&C_4^4\exp(2C_2)\exp\Big(\phi_{q+p_0+|w_j|}(x_{ww_j})-\phi_{q+p_0}(x_{ww_j})\\
&&-(\phi_{q+p_0+|w_j|}(x_{vw_j})-\phi_{q+p_0}(x_{vw_j}))\Big)\\
&\le&C_4^4\exp(2C_2+C_3).
\end{eqnarray*}
Thus by \eqref{inclusion},  \eqref{2delta} and \eqref{def-delta}, we have
$$
\frac{\mu([w]_\alpha\cap Y)}{\mu([w]_\alpha)}\le \sum_{j\ge 1} \frac{\mu([{w w_j}]_\alpha)}{\mu([w]_\alpha)}\le 2CC_4^4\exp(2C_2+C_3)\delta\le (1-C\delta)/2,
$$
which contradicts with \eqref{density1}.
\hfill $\rhd$

\medskip

Claim 2   leads to a contradiction. Indeed  we have 
$$
\underline{d}_{\mu}(\tilde x)=\liminf_{n\to\infty}\frac{\log \mu([{\tilde x|_n}]_\alpha)}{\log {\rm diam}([{\tilde x|_n}]_\alpha)}\ \ \ \text{ and }\ \ \ \underline{d}_{\mu}(\tilde y)=\liminf_{n\to\infty}\frac{\log \mu([{\tilde y|_n}]_\alpha)}{\log{\rm diam}([{\tilde y|_n}]_\alpha)}.
$$
(We note that, due to \eqref{in-out-ball},  we can use cylinders to compute the lower local dimension.)
On one hand, by the Gibbs-like property of $\mu$, we have 
\begin{equation}\label{mass-n}
\mu([\tilde x|_n]_\alpha)\sim \frac{\exp(\phi_n(\tilde x))}{\sigma_n}\ \ \ \text{ and }\ \ \ \mu([\tilde y|_n]_\alpha)\sim \frac{\exp(\phi_n(\tilde y))}{\sigma_n}.
\end{equation}
 By the bounded  covariation of $\Phi$, 
$
|\phi_n(\tilde x)-\phi_l(\tilde x)-(\phi_n(\tilde y)-\phi_l(\tilde y))|\le C_{bc}(\Phi)
$ for any $n\ge l$.
Since $l$ is a fixed number, combine with \eqref{mass-n} we conclude that 
\begin{equation}\label{sim-mu}
\mu([\tilde x|_n]_\alpha)\sim \mu([\tilde y|_n]_\alpha).
\end{equation}
On the other hand, by  \eqref{diameter} and bounded variation of $\Psi$,  
$$
{\rm diam}([\tilde x|_n]_\alpha)\sim \exp(\psi_n(\tilde x))\ \  \ \text{ and }\ \ \ {\rm diam}([\tilde y|_n]_\alpha)\sim \exp(\psi_n(\tilde y)).
$$
 By the bounded covariation  of $\Psi$, we conclude that 
\begin{equation}\label{sim-diam}
{\rm diam}([\tilde x|_n]_\alpha)\sim {\rm diam}([\tilde y|_n]_\alpha).
\end{equation}
Combine \eqref{sim-mu} and \eqref{sim-diam} we get $\underline{d}_{\mu}(\tilde x)=\underline{d}_{\mu}(\tilde y).$ However since $\tilde x\in X$ and $\tilde y\in Y$, we also have $\underline{d}_{\mu}(\tilde x)\le d_1<d_2\le \underline{d}_{\mu}(\tilde y)$, which is a contradiction.

   Thus we conclude that $\mu$ is exact lower dimensional. 

The same proof shows that $\mu$ is also  exact upper dimensional.
\hfill $\Box$

\section{Structure of the spectrum of Sturm Hamiltonian }\label{sec-structure}

In this section, as  a preparation for the proof of the main result, we study the structure of  the spectrum of Sturm Hamiltonian.

 \subsection{Facts on continued fraction and Sturm sequence}
 
Recall that Gauss map $G:(0,1)\to (0,1)$ is defined as $ G(\alpha):=1/\alpha-[1/\alpha]$.
 Assume $\alpha$ has continued fraction expansion $[0;a_1,a_2,\cdots]. $
Let   $p_n/q_n (n>0)$ be the $n$-th approximation  of   $\alpha$, 
given by:
\begin{eqnarray}
\nonumber p_{-1}=1,&\quad p_0=0,&\quad p_{n+1}=a_{n+1} p_n+p_{n-1},\ n\ge 0,\\
\label{recur-q} q_{-1}=0,&\quad q_0=1,&\quad q_{n+1}=a_{n+1} q_n+q_{n-1},\ n\ge 0.
\end{eqnarray}
We also use the notations $a_n(\alpha), p_n(\alpha)$ and $q_n(\alpha)$ to  emphasize the dependence of these quantities on $\alpha.$ 

 Fix $M\in \N$ and  define a subset of $\mathscr{B}$ as 
  $$
  \mathscr{B}_M:=\{\alpha\in [0,1]\setminus \Q: \alpha=[0;a_1,a_2,\cdots] \text{ with } a_n\le M\}.
  $$
  It is seen that $\mathscr{B}=\bigcup_{M\in \N}\mathscr{B}_M.$
The following Lemma will be useful later.

\begin{lem}\label{a-a-q-n}
(i) Given $\alpha\in [0,1]\setminus\Q.$ Then for any $n,m\in \N,$
$$
q_n(\alpha)q_m(G^n(\alpha))\le q_{n+m}(\alpha)\le 2q_n(\alpha)q_m(G^n(\alpha)).
$$

(ii) If $\alpha\in \mathscr{B}_M$, then $q_n(\alpha)\le (M+1)^n.$
\end{lem}

\begin{proof}
(i)\ 
Write $e_1=(1,0),   e_2=(0,1).$ For any $\beta\in [0,1]\setminus\Q$ and $n\in \N$, write 
 $$
   \ B_n(\beta)=
 \begin{pmatrix}
a_n(\beta)&1\\
1&0
\end{pmatrix}.
 $$
 By \eqref{recur-q} we have 
\begin{equation}\label{partial-quotient}
(q_n(\beta),q_{n-1}(\beta))^t=B_n(\beta)\cdots  B_1(\beta)\cdot e_1^t.
\end{equation}
 Since $B_1(\beta)\cdot e_2^t=e_1^t$ and $a_{n-1}(G(\beta))=a_n(\beta)$, we have 
 \begin{equation}\label{partial-quotient2}
(q_{n-1}(G(\beta)),q_{n-2}(G(\beta)))^t=B_{n}(\beta)\cdots B_{1}(\beta)\cdot e_2^t.
\end{equation}

Write $\hat \beta=G^n(\alpha)$. By \eqref{partial-quotient} and \eqref{partial-quotient2}, we have 
 \begin{eqnarray*}
q_{n+m}(\alpha)&=&e_1\cdot B_{n+m}(\alpha)\cdots B_{n+1}(\alpha)\cdot B_{n}(\alpha)\cdots B_{1}(\alpha)\cdot e_1^t\\
&=&e_1\cdot B_{m}(\hat \beta)\cdots B_{1}(\hat\beta)\cdot B_{n}(\alpha)\cdots B_{1}(\alpha)\cdot e_1^t\\
&=&e_1\cdot B_{m}(\hat \beta)\cdots B_{1}(\hat\beta)\cdot (q_n(\alpha),q_{n-1}(\alpha))^t\\
&=& q_n(\alpha)e_1\cdot B_{m}(\hat\beta)\cdots B_{1}(\hat \beta)\cdot e_1^t+\\
&&q_{n-1}(\alpha)e_1\cdot B_{m}(\hat\beta)\cdots B_{1}(\hat\beta)\cdot e_2^t\\
&=&q_n(\alpha) q_m(\hat\beta)+q_{n-1}(\alpha)q_{m-1}(G(\hat\beta)).
\end{eqnarray*}
By induction, it is ready to show that $0<q_{n-1}(\alpha)\le q_n(\alpha)$ and $0<q_{m-1}(G(\hat\beta))\le q_{m}(\hat\beta).$
Then the result follows.

(ii)\ It follows from \eqref{recur-q} and induction.
\end{proof}

For $\alpha\in [0,1]\setminus \Q$,  the {\it standard } Strum sequence $S_n(\alpha)$ is defined by 
$$
S_n(\alpha):=\chi_{[1-\alpha,1)}(n\alpha\pmod1),\ \ \ (n\in\Z).
$$ 
\begin{lem}\label{basic-sturm}
If $a_i(\alpha)=a_i(\beta)$ for $i=1,\cdots,n$, then $q_n(\alpha)=q_n(\beta)=:q_n$. Moreover
$
S_i(\alpha)=S_i(\beta)
$ 
for  $1\le i\le q_n.$
\end{lem}

\begin{proof}
The first statement follows from \eqref{recur-q}. See  \cite{AS,Lo} for a proof of the second statement. 
\end{proof}

 \subsection{Structure and coding of the spectrum}

 \subsubsection{The structure of the spectrum} 
 
 We describe  the structure of the spectrum
$\Sigma_{\alpha,\lambda}$ for fixed  $\alpha\in[0,1]\setminus\Q$ and $\lambda>0$, for more
details, we refer to \cite{T,BIST,R,LW}.

Recall that the Sturm potential is given by 
$
v_n=\lambda\chi_{[1-\alpha,1)}(n\alpha+\theta\pmod 1).
$
Since $\Sigma_{\alpha,\lambda}$ is independent of the phase $\theta$, in the rest of the paper we will take $\theta=0.$ Thus $v_n=\lambda S_n(\alpha).$ Assume $\alpha$ has continued fraction expansion $[0;a_1,a_2,\cdots]$ and define $q_n$ by \eqref{recur-q}.
 For any $n\geq1$ and $x\in\mathbb{R}$, the transfer matrix $M_n(x)$
over $q_n$ sites is defined by
$${\mathbf M}_n(x):=
\left[\begin{array}{cc}x-v_{q_n}&-1\\ 1&0\end{array}\right]
\left[\begin{array}{cc}x-v_{q_n-1}&-1\\ 1&0\end{array}\right]
\cdots
\left[\begin{array}{cc}x-v_1&-1\\ 1&0\end{array}\right],
$$
By convention we  take
$$
\begin{array}{l}
{\mathbf M}_{-1}(x)= \left[\begin{array}{cc}1&-\lambda\\
0&1\end{array}\right]
\end{array}
\ \ \ \text{ and }\ \ \ 
\begin{array}{l}
{\mathbf M}_{0}(x)= \left[\begin{array}{cc}x&-1\\
1&0\end{array}\right].
\end{array}
$$

\smallskip

For $n\ge0$, $p\ge-1$, let $h_{(n,p)}(x):=\tr {\mathbf M}_{n-1}(x) {\mathbf M}_n^p(x)$ and
$$
\sigma_{(n,p)}:=\{x\in\mathbb{R}:|h_{(n,p)}(x)|\leq2\},
$$
 where  $\tr M$
stands for the trace of the matrix $M$. $\sigma_{(n,p)}$ is made of finitely many bands.
Moreover, for any $n\ge 0,$
$$
(\sigma_{(n+2,0)}\cup\sigma_{(n+1,0)})\subset
(\sigma_{(n+1,0)}\cup\sigma_{(n,0)})\ \text{ and }\ \Sigma_{\alpha,\lambda}=\bigcap_{n\ge0}(\sigma_{(n+1,0)}\cup\sigma_{(n,0)}).
$$
The intervals in  $\sigma_{(n,p)}$ are called  {\em bands}.  For any band 
$B$ of $ \sigma_{(n,p)}$,  $h_{(n,p)}(x)$ is monotone on
$B$ and
$h_{(n,p)}(B)=[-2,2].$
We call $h_{(n,p)}$ the {\em generating polynomial} of $B$ and denote it by $h_B:=h_{(n,p)}$.

$\{\sigma_{(n+1,0)}\cup\sigma_{(n,0)}:n\ge 0\}$ form a covering of $\Sigma_{\alpha,\lambda}$.
However there are some repetitions between $\sigma_{(n,0)}\cup\sigma_{(n-1,0)}$
and $\sigma_{(n+1,0)}\cup\sigma_{(n,0)}$. When $\lambda>4,$
it is possible to choose a covering of $\Sigma_{\alpha,\lambda}$ elaborately such that
we can get rid of these repetitions, as we will describe in the follows:

\begin{defi}{\rm (\cite{R,LW})}\label{def-4.3}
For $\lambda>4$, $n\ge0$, define three types of bands as:

$(n,I)$-type band: a band of $\sigma_{(n,1)}$ contained in a
band of $\sigma_{(n,0)}$;

$(n,II)$-type band: a band of $\sigma_{(n+1,0)}$ contained
in a band of $\sigma_{(n,-1)}$;

$(n,III)$-type band: a band of $\sigma_{(n+1,0)}$ contained
in a band of $\sigma_{(n,0)}$.
\end{defi}

All three types of bands actually occur and they are disjoint. We call these bands {\em
spectral generating bands of order $n$}. Note that  there 
are only two spectral generating bands of order $0$, 
one is $\sigma_{(0,1)}=[\lambda-2,\lambda+2]$ with  generating polynomial $h_{(0,1)}=x-\lambda$ and type $(0,I)$, the other is  $\sigma_{(1,0)}=[-2,2]$ with generating polynomial  $h_{(1,0)}=x$ and type $(0,III)$.

For any $n\ge0$, denote by $\B_n$ the set of spectral
generating bands of order $n$, then the intervals in $\B_n$ are disjoint.
Moreover (\cite{R,LW})
\begin{itemize}
\item $(\sigma_{(n+2,0)}\cup\sigma_{(n+1,0)})\subset
\bigcup_{B\in\B_n}B
\subset (\sigma_{(n+1,0)}\cup\sigma_{(n,0)})$,
thus
$$
\Sigma_{\alpha,\lambda}=\bigcap_{n\ge0}
\bigcup_{B\in\B_n}B.
$$
\item
any $(n,I)$-type band contains only one band in $\B_{n+1}$, which is of  $(n+1,II)$-type.
\item
any $(n,II)$-type band contains $2a_{n+1}+1$ bands in $\B_{n+1}$,
$a_{n+1}+1$ of which are of  $(n+1,I)$-type and $a_{n+1}$ of which are of  $(n+1,III)$-type.
\item
any $(n,III)$-type band contains $2a_{n+1}-1$ bands in $\B_{n+1}$,
$a_{n+1}$ of which are of  $(n+1,I)$-type and $a_{n+1}-1$ of which are of  $(n+1,III)$-type.
\end{itemize}

Thus $\{\B_n\}_{n\ge0}$ form a natural
covering of the spectrum $\Sigma_{\alpha,\lambda}$ (\cite{LW05,LPW07}). 

\begin{rem}\label{B-n}
{\rm
From the definition, $\B_0=\{[\lambda-2,\lambda+2], [-2,2]\}$. For $n\ge 1,$   $\B_n$ should depend on $\alpha$, i.e., $\B_n=\B_n(\alpha)$. On the other hand, by Definition \ref{def-4.3}, Lemma \ref{basic-sturm} and the fact that $v_k=\lambda S_k(\alpha),$  we conclude that $\B_n$ indeed only depends on $\vec{a}=a_1\cdots a_n$, where $\alpha=[0;a_1,a_2,\cdots]$. That is,  $\B_n=\B_n(\vec{a})$.
} 
\end{rem}

\subsubsection{The coding of the spectrum}

 In the following we give a coding of  the spectrum $\Sigma_{\alpha,\lambda}$ based on \cite{R,LQW,Q}.   Here we essentially follow  \cite{Q}.
 For each $n\in\N,$ define an alphabet $\A_n$ as 
\begin{equation*}\label{alphabet}
\A_n:=\{(I,j)_n:j=1,\cdots, n+1\}\cup \{II_n\}\cup \{(III,j)_n:j=1,\cdots, n\}.
\end{equation*}
Then $\#\A_n=2n+2. $
We order the elements in $\A_n$ as 
$$
(I,1)_n<\cdots <(I,n+1)_n<II_n<(III,1)_n<\cdots< (III,n)_n.
$$
To simplify the notation, we rename  the above line as 
$$
e_{n,1}<e_{n,2}<\cdots< e_{n,2n+2}.
$$
Given $e_{n,i}\in\A_n$ and $e_{m,j}\in \A_m$, we call $e_{n,i}e_{m,j}$  {\it admissible}, denote by $e_{n,i}\to e_{m,j},$ if 
\begin{eqnarray*}
(e_{n,i},e_{m,j})&\in& \{((I,k)_n,II_m): 1\le k\le n+1\}\cup \\
&&\{(II_n,(I,l)_m): 1\le l\le m+1\}\cup \\
&&\{(II_n,(III,l)_m): 1\le l\le m\}\cup\\
&&\{((III,k)_n,(I,l)_m):1\le k\le n,\ 1\le l\le m\}\cup \\
&&\{((III,k)_n,(III,l)_m):1\le k\le n, 1\le l\le m-1\}.
\end{eqnarray*}
For pair $(\A_n, \A_m),$ we define the incidence matrix $A_{nm}= (a_{ij})$ of size $(2n+2)\times(2m+2)$ as 
\begin{equation}\label{A_ij}
a_{ij}=
\begin{cases}
1& e_{n,i}\to e_{m,j},\\
0& \text{ otherwise}.
\end{cases}
\end{equation}
When $n=m,$ we write $A_n:=A_{nn}.$
For any $n\in \N$, 
we define an auxiliary matrix as follows:
$$
\hat A_n=
\left[\begin{array}{ccc}
0&1&0\\
n+1&0&n\\
n&0&n-1
\end{array}\right].
$$

Let us show that  $\{A_{nm}: n,m\ge1\}$ defined by \eqref{A_ij} is strongly primitive:

\begin{prop}\label{sturm-primitive}
For any $k\ge 6$ and any  sequence $a_1\cdots a_k\in \N^k$, the matrix 
$A_{a_1a_2}A_{a_2a_3}\cdots A_{a_{k-1}a_k}$ is positive. 
\end{prop}

\begin{proof}
At first we show the result for $k=6.$ Fix any sequence $a_1\cdots a_6\in \N^6$
and  any  $e\in \A_{a_1}$ and $\hat e\in \A_{a_6}$, we only need  to show that there is an admissible word $e_1\cdots e_6\in \prod_{j=1}^6\A_{a_j}$ such that $e_1=e$ and $e_6=\hat e.$  We construct the desired admissible word as follows.

If $e=(I,j)_{a_1}$ and $\hat e=(I,l)_{a_6}$ or $(III,l)_{a_6}$, take
$$
e_1\cdots e_6=e\ II_{a_2}(III,1)_{a_3}(I,1)_{a_4}II_{a_5} \hat e.
$$ 

If $e=(I,j)_{a_1}$ and $\hat e=II_{a_6}$, take
$$
e_1\cdots e_6=e\ II_{a_2}(I,1)_{a_3}II_{a_4}(I,1)_{a_5} \hat e.
$$ 

If $e=II_{a_1}$ and $\hat e=(I,l)_{a_6}$ or $(III,l)_{a_6}$, take
$$
e_1\cdots e_6=e\ (I,1)_{a_2}II_{a_3}(I,1)_{a_4}II_{a_5} \hat e.
$$ 

If $e=II_{a_1}$ and $\hat e=II_{a_6}$, take
$$
e_1\cdots e_6=e\ (III,1)_{a_2}(I,1)_{a_3}II_{a_4}(I,1)_{a_5} \hat e.
$$ 

If $e=(III,j)_{a_1}$ and $\hat e=(I,l)_{a_6}$ or $(III,l)_{a_6}$, take
$$
e_1\cdots e_6=e\ (I,1)_{a_2}II_{a_3}(I,1)_{a_4}II_{a_5} \hat e.
$$ 

If $e=(III,j)_{a_1}$ and $\hat e=II_{a_6}$, take
$$
e_1\cdots e_6=e\ (I,1)_{a_2}II_{a_3}(III,1)_{a_4}(I,1)_{a_5} \hat e.
$$ 

For general $k$, we note that for any $n,m\in \N$, each column and each array of $A_{nm}$ has at least one $1$. Then the result follows easily from the case $k=6$.
\end{proof}

We also define $\A_0:=\{I, III\}$ and write $e_{0,1}=I$ and $e_{0,2}=III$.  Fix $m\in \N.$ Given $e_{0,i}\in \A_0$ and $e_{m,j}\in \A_{m}$, we call $e_{0,i}e_{m,j}$  {\it admissible}, denote by $e_{0,i}\to e_{m,j},$ if 
\begin{eqnarray*}
(e_{0,i},e_{m,j})&\in& \{(I,II_m)\}\cup 
\{(III,(I,l)_m):1\le l\le m\}\cup \\
&&\{(III,(III,l)_m):1\le l\le m-1\}.
\end{eqnarray*}
For pair $(\A_0, \A_m),$ we define the incidence matrix $A_{0m}= (a_{ij})$ of size $2\times(2m+2)$ as 
$$
a_{ij}=
\begin{cases}
1& e_{0,i}\to e_{m,j},\\
0& \text{ otherwise}.
\end{cases}
$$

For any $\vec{a}=a_1\cdots a_n\in\N^n$ and any $0\le m\le n,$ define 
$$
\Omega_m^{(\vec{a})}:=\{w_0\cdots w_m\in \A_0\times\prod_{j=1}^m\A_{a_j}: w_j\to w_{j+1} \}.
$$

If $\alpha=[0;a_1,a_2,\cdots]$, define a symbolic space  $\Omega^{(\alpha)}$ as    
\begin{eqnarray*}
\Omega^{(\alpha)}=\{x_0x_1x_2\cdots \in \A_0\times\prod_{j=1}^\infty \A_{a_j}: x_j\to x_{j+1} \}.
\end{eqnarray*}
For $x\in \Omega^{(\alpha)}$, we write $x|_n:=x_0\cdots x_n.$ More generally we write $x[n, n+k]$ for $x_n\cdots x_{n+k}.$ 
Define $\Omega^{(\alpha)}_{n}:=\{x|_n: x\in \Omega^{(\alpha)}\}$ and $\Omega^{(\alpha)}_{\ast}=\bigcup_{n} \Omega^{(\alpha)}_{n}.$ It is seen that $\Omega^{(\alpha)}_{n}=\Omega^{(\vec{a})}_{n}$, where $\vec{a}=a_1\cdots a_n.$   Given $w=w_0\cdots w_n\in \Omega^{(\alpha)}_{n}$, denote by $|w|$ the length of $w$, then $|w|=n+1$.  If the last letter $w_n=(T,j)_{a_n}$, then we write $t_w=T,$ and say that  $w$ has {\it  type } $T.$  If $w=uw^\prime$, then we say $u$ is a {\it prefix } of $w$ and denote by $u\prec w.$
 For $x,y\in \Omega^{(\alpha)}$, we denote by $x\wedge y$ the maximal common prefix of $x$ and $y.$
Given $w\in \Omega^{(\alpha)}_{n},$
we define the cylinder 
$$
[w]^\alpha:=\{x\in \Omega^{(\alpha)}: x|_n=w\}.
$$

 In the following we explain that $\Omega^{(\alpha)}$ is a coding of the spectrum $\Sigma_{\alpha,\lambda}$. 
 
 At first we fix $\vec{a}=a_1\cdots a_n\in \N^n$ and code the intervals in $\B_n(\vec{a})$  by $\Omega_n^{(\vec{a})}$ (see Remark \ref{B-n} for the notation $\B_n(\vec{a})$). Write $\B_0(\vec{a}):=\B_0$. For any $1\le m\le n$, write $\B_m(\vec{a}):=\B_m(a_1\cdots a_m)$.
 Define 
$B_{I}:=[\lambda-2,\lambda+2]$  and  $B_{III}:=[-2,2]$, then 
$
\B_0(\vec{a})=\B_0=\{B_w: w\in \Omega_0^{(\vec{a})}\}.
$
Assume $B_w$ is defined for any $w\in  \Omega_{m-1}^{(\vec{a})}$.
Now for any $w\in\Omega_m^{(\vec{a})}$, write $w=w^\prime e=w^\prime(T,j)_{a_m}$. Then 
define $B_w$ to be the unique $j$-th  $(m,T)$ type band in $\B_m(\vec{a})$ which is contained in $B_{w^\prime}$.  Then 
$
\B_m(\vec{a})=\{B_w: w\in\Omega_m^{(\vec{a})}\}.
$
By induction, we code all the bands in $\B_n(\vec{a})$ by $\Omega_n^{(\vec{a})}.$ 

 Now we can define a natural projection $\pi_\alpha:\Omega^{(\alpha)}\to \Sigma_{\alpha,\lambda}$ as 
\begin{equation}\label{pi-alpha}
\pi_\alpha(x):=\bigcap_{n\ge 0} B_{x|_n}.
\end{equation} 
 It is seen that $\pi_\alpha$ is a bijection, thus  $\Omega^{(\alpha)}$  is a coding of $\Sigma_{\alpha,\lambda}$.

\begin{rem}
 {\rm
$\Omega^{(\alpha)}$ is not exactly the symbolic space $\Omega_\alpha$ defined in Sect. \ref{sec-space-potential}, indeed  it is an  ``augment" of $\Omega_\alpha$ with alphabet sequence $\{\A_0, \A_{a_1},\A_{a_2}\cdots\}$ and  incidence matrix sequence $\{A_{0a_1}, A_{a_1a_2},\cdots\}$.  However, since $0$ only appears at the first place, by checking the proof, all the results developed in Sect. \ref{sec-gibbs-like} still holds on the symbolic space $\Omega^{(\alpha)}.$ We denote the class of  regular (negative) potentials on $\Omega^{(\alpha)}$  with bounded variation and covariation by $\mathcal{F}^\alpha (\mathcal{F}^\alpha_-)$.
 }
\end{rem}

We also need the following facts later.
\begin{lem}\label{auxi}
(i) For any sequence $a_1\cdots a_5\in \{1,\cdots,M\}^5$, we have
$$
\begin{pmatrix}
1&1&1\\
1&1&1\\
1&1&1
\end{pmatrix}
\le \hat  A_{a_1}\cdots \hat A_{a_5}\le 81(M+1)^5
\begin{pmatrix}
1&1&1\\
1&1&1\\
1&1&1
\end{pmatrix}.
$$ 

(ii) If $\alpha=[0;a_1,a_2,\cdots]$, then 
$$
\#\Omega^{(\alpha)}_n=(1,0,1)\cdot \hat A_{a_1}\cdots \hat A_{a_n}\cdot (1,1,1)^t.
$$
\end{lem}

\begin{proof}
(i) Notice that if $1\le a\le M,$ then $\hat A_1\le \hat A_a\le (M+1)J$, where $J$ is the matrices with all entries equal to $1$. Then the result follows by direct computation.

(ii) It follows from the definition and induction.
\end{proof}

\subsection{Other useful facts of Sturm Hamiltonian}

We collect some useful results of Sturm Hamiltonian for later use. 
Fix $\alpha=[0;a_1,a_2,\cdots]$ and consider the operator $H_{\alpha,\lambda,0}$.
We write $h_k(x):=h_{k+1,0}(x)={\tr M_k(x)}$ and
\begin{equation}\label{sigmak}
\sigma_k:=\sigma_{(k+1,0)}=\{x\in\R : h_k(x)|\le 2\}.
\end{equation} 
We note that $\sigma_k$ consists of $q_k(\alpha)$ disjoint intervals.

The following two lemmas are implicitly   proven in \cite{R}:

\begin{lem}\label{sigma-k}
$\sigma_k=\{B\in \B_k(\alpha): B \text{ is of  type } (k,II) \text{ or } (k,III)\}$. 
\end{lem}

Define $v^I=(1,0,0), v^{II}=(0,1,0), v^{III}=(0,0,1)$ and $v_\ast=(0,1,1)^T.$

\begin{lem}\label{number-n-m}
Given $w\in  \Omega_n^{(\alpha)}$,  then 
\begin{equation*}\label{num}
\#\{u\in \Omega_{n+m}^{(\alpha)}:w\prec u, t_u=II,III\}=v^{t_w}\cdot \hat A_{a_{n+1}}\cdots \hat A_{a_{n+m}} \cdot v_\ast.
\end{equation*}
\end{lem}

The following theorem is \cite{FLW} Theorem 5.1, see also \cite{LQW} Theorem 3.3.
\begin{thm}[Bounded covariation]\label{bco}
Assume $\lambda>20$ and  $\alpha, \beta\in \mathscr{B}_M$. Then there exists 
constant $\eta>1$ only depending on $\lambda$ and $M$  such
that if $w, wu\in  \Omega_\ast^{(\alpha)}$ and $\tilde w, \tilde wu\in  \Omega_\ast^{(\beta)}$,
then  
$$\eta^{-1} \frac{|{B}_{wu}|}{|{B}_{w}|}\le
\frac{|B_{\widetilde wu}|}{|B_{\widetilde w}|}\le \eta
\frac{|{B}_{wu}|}{|{B}_{w}|}.$$
\end{thm}
We remark that  in \cite{FLW},  only  the case $\alpha=\beta$ is considered. However by checking the proof, the same argument indeed shows the  stronger result as stated in Theorem \ref{bco}. 

The following lemma is a direct consequence of \cite{FLW} Corollary 3.1:
\begin{lem}
Let $\lambda>20$ and $\alpha\in \mathscr{B}_M$. 
Then there exists 
constant $0<c=c(\lambda,M)<1$ such that for any $w\in  \Omega_n^{(\alpha)},$
\begin{equation}\label{band-length}
c|B_{w|_{n-1}}|\le|B_{w}|\ \  \text{ and }\ \  |B_w|\le 2^{2-n}.
\end{equation}
\end{lem}

 The following result is also useful later.
 
 \begin{lem}\label{compare-q-n}
 $q_n(\alpha)\le \#\Omega_n^{(\alpha)}\le 5q_n(\alpha)$.
 \end{lem}
 
 \begin{proof}
 By the definition of $\B_n(\alpha),$ we know that 
 \begin{eqnarray*}
 \B_n(\alpha)&=&\{B_w: w\in\Omega_n^{(\alpha)} \}\\
 &=&\{B_w: w\in\Omega_n^{(\alpha)}; t_w=II,III\}\cup\{B_w: w\in\Omega_n^{(\alpha)}; t_w=I\} \\
 &=&\{B_w: w\in\Omega_n^{(\alpha)}; t_w=II,III\}\cup\\
 &&\{B_{u(I,j)_{a_n}}: u\in\Omega_{n-1}^{(\alpha)}; t_u=II, 1\le j\le a_n+1\}\cup\\
 &&\{B_{u(I,j)_{a_n}}: u\in\Omega_{n-1}^{(\alpha)}; t_u=III, 1\le j\le a_n\}.
 \end{eqnarray*}
  Recall that $\sigma_n$ is the $n$-th approximation of $\Sigma_{\alpha,\lambda}$ (see \eqref{sigmak}), which is made of $q_n(\alpha)$ bands.
 Combine with Lemma \ref{sigma-k} and \eqref{recur-q}, we have 
 $$
 q_n(\alpha)\le \#\B_n(\alpha)=\#\Omega_n^{(\alpha)}\le q_n(\alpha)+4a_nq_{n-1}(\alpha)\le 5q_n(\alpha).
 $$
 \end{proof}

\section{Exact-dimensional property of DOS}\label{sec-proof-main}

In this section, we apply the result in Sect. \ref{sec-gibbs-like} to prove Theorem \ref{main}. At first, we construct a potential $\Psi^\alpha\in \mathcal{F}_{-}^\alpha$ and define  the related weak Gibbs metric $d_\alpha=d_{\Psi^\alpha}$ such that  $\pi_\alpha$ is a bi-Lipschitz homeomorphism between $(\Omega^{(\alpha)},d_\alpha)$
and $(\Sigma_{\alpha,\lambda},|\cdot|)$. Then we construct a potential $\Phi^\alpha\in \mathcal{F}^\alpha$ such that $(\pi_\alpha)_\ast(\mu_\alpha)\asymp \NN_{\alpha,\lambda}$, where $\mu_\alpha$ is a Gibbs-like measure of $\Phi^\alpha.$ From this, we conclude that $\NN_{\alpha,\lambda}$ is exact upper and lower dimensional. In the last subsection, we construct a frequency $\alpha\in \mathscr{B}_2$, such that $\NN_{\alpha,\lambda}$ is not exact-dimensional when $\lambda$ is large enough.

Throughout the first two subsections, we always fix $\alpha\in \mathscr{B}$ and $\lambda>20.$  Choose $M\in \N$  such that $\alpha\in \mathscr{B}_M.$


\subsection{Weak Gibbs metric on $\Omega^{(\alpha)}$}

Define  $\Psi^\alpha=\{\psi_n^\alpha: n\ge 1\}$ on $\Omega^{(\alpha)}$ as 
 \begin{equation}\label{def-Psi}
 \psi_n^\alpha(x):=\log |B_{x|_{n}}|.
 \end{equation}

\begin{lem}  \label{bdvar-Psi}
$\Psi^\alpha\in  \F^\alpha_-$  with related constants  $C_{\rm rg}(\Psi^\alpha), C_{\rm bv}(\Psi^\alpha), C_{\rm bc}(\Psi^\alpha)$  only depending on $\lambda$ and $M$. 
\end{lem}

\begin{proof}
Fix any $n,k\in \N$.  Then 
$\psi^\alpha_{n+k}(x)-\psi^\alpha_n(x)=\log(|B_{x|_{n+k}}|/|B_{x|n}|).$ Since $\Omega^{(\alpha)}_0=\{I,III\}$ and $\Omega^{(\alpha)}_1$ contains at least one word of type $II$, there exists $u\in \Omega^{(\alpha)}_0\cup\Omega^{(\alpha)}_1$ such that  $u$ and $x|_n$ have the same type. Write $w=x_{n+1}\cdots x_{n+k}$, by \eqref{band-length}, there exists constant $0<c<1$ only depending on $\lambda$ and $M$ such that  
$$
4c\le |B_u|\le 4\ \  \ \text{ and }\ \ \ 4c^{k+1}\le |B_{uw}|\le 4\cdot 2^{-k}.
$$
Then by Theorem \ref{bco}, there exists  constant $\eta>1$ only depending on $\lambda$ and $M$ such that 
\begin{eqnarray*}
|\psi^\alpha_{n+k}(x)-\psi^\alpha_n(x)|&=&\left|\log\left(\frac{|B_{x|_{n+k}}|}{|B_{x|n}|}/\frac{|B_{uw}|}{|B_{u}|}\right)+\log \frac{|B_{uw}|}{|B_{u}|}\right|\\
&\le&\log \eta+|\log|B_{uw}||+|\log |B_u||\\
&\le& \log \eta+4\log 2 +k(\log 2-3\log c)\\
&\le &k(5\log 2-3\log c+\log \eta).
\end{eqnarray*}
Hence $\Psi^\alpha$ is regular with $C_{\rm rg}(\Psi^\alpha)=5\log 2-3\log c+\log \eta.$

From the definition, if $x|_n=y|_n$, then $ \psi_n^\alpha(x)= \psi_n^\alpha(y).$ Hence $\Psi^\alpha$ has bounded variation with $C_{\rm bv}(\Psi^\alpha)=0.$

Now assume $w=uv,\tilde w=\tilde u v\in \Omega^{(\alpha)}_{\ast}$ and  $x\in [w]^\alpha, \tilde x\in [\tilde w]^\alpha$.  By Theorem \ref{bco}, we have 
\begin{eqnarray*} 
|(\psi^\alpha_{|w|}(x)-\psi^\alpha_{|u|}(x))-( \psi^\alpha_{|\tilde w|}(\tilde x)-\psi^\alpha_{|\tilde u|}(\tilde x))|=\left|\log \left(\frac{|B_w|}{|B_u|}/\frac{|B_{\tilde w}|}{|B_{\tilde u}|}\right)\right|\le \log \eta.
\end{eqnarray*}
Thus $\Psi^\alpha$ has bounded covariation with $C_{\rm bc}(\Psi^\alpha)=\log \eta.$

Thus $\Psi^\alpha\in \mathcal{F}^\alpha$ and $C_{\rm rg}(\Psi^\alpha), C_{\rm bv}(\Psi^\alpha), C_{\rm bc}(\Psi^\alpha)$ only depend on $\lambda$ and $M.$

   Since $B_{x|_{n+1}}$ is a subinterval of $B_{x|_{n}},$ combine with \eqref{band-length}, we have 
   $$
   \psi^\alpha_{n+1}(x)=\log |B_{x|_{n+1}}|\le \log |B_{x|_{n}}|= \psi^\alpha_{n}(x)\le (2-n)\log 2.
   $$
   Thus  $\Psi^\alpha $ is negative.
\end{proof}

Define the weak Gibbs metric $d_\alpha:=d_{\Psi^\alpha}$ on $\Omega^{(\alpha)},$ then 
for any $x,y\in \Omega^{(\alpha)}$,  
 \begin{equation*}
{ {d}_\alpha}(x,y)=|B_{x\wedge y}|.
\end{equation*}
Recall that $\pi_\alpha$ is defined by \eqref{pi-alpha}.

\begin{prop}\label{bi-lip}
$\pi_\alpha: (\Omega^{(\alpha)}, d_\alpha)\to (\Sigma_{\alpha,\lambda},|\cdot|)$ is bi-Lipschitz.
\end{prop}

To prove this property, we do some preparation.  Write 
$
{\rm Co}(\Sigma_{\alpha,\lambda})\setminus \Sigma_{\alpha,\lambda}=\bigcup_i G_i,
$ 
where ${\rm Co}(\Sigma_{\alpha,\lambda})$ is the convex hull of $\Sigma_{\alpha,\lambda}.$
Each $G_i$ is called a {\it  gap} of the spectrum. A gap $G$ is called of {\it order} $n$, if $G$ is covered by some band in $\B_n $ but not covered by any band in $\B_{n+1}.$ Denote by ${\mathcal G}_n$  the set of gaps of order $n$. For any $G\in {\mathcal G}_n$, let $B_G$ be the unique band in $\B_n$ which contains $G.$ The following lemma is proven in \cite{Q}:
 
 \begin{lem} \label{gap-esti}
 There exists a constant $C=C(M,\lambda)\in (0,1)$ such that  for any gap $G\in \bigcup_{n\ge 0}{\mathcal G}_n$ we have 
$
|G|\ge C|B_G|.
$
\end{lem}

\noindent {\bf Proof of Proposition \ref{bi-lip}. }\ 
Given $x,y\in  \Omega^{(\alpha)}.$ Assume $x|_n=y|_n$ and $x_{n+1}\ne y_{n+1},$ then 
$
 d_\alpha(x,y) =|B_{x|_n}|.
$
Write $\hat x:=\pi_{\alpha}(x)$ and $\hat y:=\pi_{\alpha}(y)$. By the definition of $\pi_\alpha,$ we have $\hat x,\hat y\in B_{x|_n}$, consequently 
$$
|\hat x-\hat y|\le |B_{x|_n}|= d_\alpha(x,y).
$$

On the other hand, since $x_{n+1}\ne y_{n+1},$ there is a gap $G$ of order $n$ which is contained in the open interval $(\hat x,\hat y).$  By Lemma \ref{gap-esti}, there exists a constant $C=C(M,\lambda)>0$ such that 
$$
|\hat x-\hat y|\ge |G|\ge C |B_G|=C|B_{x|_n}|=C  d_\alpha(x,y).
$$
Thus $\pi_{\alpha}$ is  bi-Lipschitz.
\hfill $\Box$


\subsection{ DOS is  exact upper and lower dimensional}

Define $\Phi^\alpha=\{\phi_n^\alpha: n\ge 1\}$ on $\Omega^{(\alpha)}$ as
\begin{equation}\label{def-Phi}
\phi^\alpha_n(x):=\log q_n(\alpha).
\end{equation}

\begin{prop}\label{F-alpha}
$\Phi^\alpha\in \mathcal{F}^\alpha$ with 
$$
C_{\rm rg}(\Phi^\alpha)=2\log(M+1), \ \ C_{\rm bv}(\Phi^\alpha)=0\ \ \text{ and }\ \ 
C_{\rm bc}(\Phi^\alpha)=\log 2.
$$
\end{prop}

\begin{proof}
For any $n,m\ge 1$, by Lemma \ref{a-a-q-n} and the fact that $G^n(\alpha)\in \mathscr{B}_M$,  
\begin{eqnarray*}
|\phi^\alpha_n(x)-\phi^\alpha_{n+m}(x)|&=&|\log q_n(\alpha)-\log q_{n+m}(\alpha)|\le \log 2+\log|q_m(G^n(\alpha))|\\
&\le& \log 2+ m\log (M+1)\le 2m\log (M+1).
\end{eqnarray*}
Thus $\Phi^\alpha$ is regular with $C_{\rm rg}(\Phi^\alpha)=2\log(M+1)$.

From the definition, $\Phi^\alpha$ has bounded variation with $C_{\rm bv}(\Phi^\alpha)=0.$

Now assume $w=uv,\tilde w=\tilde u v\in \Omega^{(\alpha)}_{\ast}$ and $x\in [w]^\alpha, \tilde x\in [\tilde w]^\alpha$. Write $n=|u|, \tilde n=|\tilde u|$ and $m=|v|.$
Similar computation as above shows that 
\begin{eqnarray*}
\log q_m(G^n(\alpha))& \le\phi^\alpha_{|w|}(x)-\phi^\alpha_{|u|}(x)&\le \log 2+\log q_m(G^n(\alpha)),\\
\log q_m(G^{\tilde n}(\alpha))& \le\phi^\alpha_{|\tilde w|}(x)-\phi^\alpha_{|\tilde u|}(x)&\le \log 2+\log q_m(G^{\tilde n}(\alpha)).
\end{eqnarray*}
Since $w$ and $\tilde w$ have the same suffix of length $m$, we have 
$$
G^n(\alpha)=[0; v_1,v_2,\cdots, v_m, b_{m+1}\cdots ];\ G^{\tilde n}(\alpha)=[0; v_1,v_2,\cdots, v_m, \tilde b_{m+1}\cdots ].
$$
Hence $q_m(G^{ n}(\alpha))=q_m(G^{\tilde n}(\alpha)).$ As a consequence, we have
\begin{equation*} 
|(\phi^\alpha_{|w|}(x)-\phi^\alpha_{|u|}(x))-( \phi^\alpha_{|\tilde w|}(\tilde x)-\phi^\alpha_{|\tilde u|}(\tilde x))|\le \log 2.
\end{equation*}
Thus $\Phi^\alpha$ has bounded covariation with $C_{\rm bc}(\Phi^\alpha)=\log 2.$ 

By the definition, $\Phi^\alpha\in \mathcal{F}^\alpha.$
\end{proof}

Let $\mu_\alpha$ be a Gibbs-like measure of $\Phi^\alpha$.

\begin{prop}\label{strong-equiv}
$(\pi_\alpha)_\ast(\mu_\alpha)\asymp \NN_{\alpha,\lambda}$.
\end{prop}

\begin{proof}
By the definition of $\pi_\alpha$ and $\asymp$, we only need to show that 
\begin{equation}\label{sim-mu-NN}
\mu_\alpha([w]^\alpha)\sim \NN_{\alpha,\lambda}(B_w),\ \ \ \forall w\in \Omega^{(\alpha)}_{\ast}.
\end{equation}

On one hand, by Theorem \ref{exist-gibbs-like}, Proposition \ref{F-alpha},  and Lemma \ref{compare-q-n},  there exists a constant $C>1$ only depending on $M$ such that 
for any $w\in \Omega^{(\alpha)}_{n},$
\begin{equation}\label{uniform-control}
\frac{1}{Cq_n(\alpha)}\le \mu_\alpha([w]^\alpha) \le \frac{C}{q_n(\alpha)}.
\end{equation}

On the other hand, let $H_n$ be the restriction of  $H_{\alpha,\lambda,0}$ to the box $[1,q_{n}]$ with periodic boundary condition. Let  
$
\mathcal X_n=\{x_{n,1},\cdots, x_{n,q_n}\}
$ be the eigenvalues of $H_n$. Recall that $\sigma_n$ is defined by \eqref{sigmak}. It is well-known that
each band in $\sigma_n$ contains exactly one value in $\mathcal X_n$(see for example \cite{T,R}).
 Define 
$
\nu_n:=\frac{1}{q_n}\sum_{i=1}^{q_n}\delta_{x_{n,i}},
$
then  $\nu_n\to \NN_{\alpha,\lambda}$ weakly(see for example \cite{CL}).
 For any $l>n+10$,  by Lemma \ref{sigma-k} and Lemma \ref{number-n-m},  we have 
\begin{eqnarray*}
\nu_l(B_{w})&=&\sum_{u\in\Omega^{(\alpha)}_{l}, w\prec u}\nu_l (B_u)\\
&=&\frac{\# \{u\in\Omega^{(\alpha)}_{l}: w\prec u, t_u=II, III\} }{q_l(\alpha)}\\
&=&\frac{v^{t_w}\cdot \hat A_{a_{n+1}}\cdots \hat A_{a_{l}} \cdot v_\ast}{q_l(\alpha)}.
\end{eqnarray*}
By Lemma \ref{auxi} (i), 
\begin{eqnarray*}
v^{t_w}\cdot \hat A_{a_{n+1}}\cdots \hat A_{a_{n+5}}&\sim_M &(1,0,1)\cdot \hat A_{a_{n+1}}\cdots \hat A_{a_{n+5}}\\
 \hat A_{a_{l-4}}\cdots \hat A_{a_{l}}\cdot v_\ast&\sim_M&  \hat A_{a_{l-4}}\cdots \hat A_{a_{l}}\cdot (1,1,1)^t.
\end{eqnarray*}
Consequently, by Lemma \ref{auxi} (ii), Lemma \ref{compare-q-n} and Lemma \ref{a-a-q-n} (i), we have 
\begin{eqnarray*}
\nu_l(B_{w})&\sim&\frac{(1,0,1)\cdot \hat A_{a_{n+1}}\cdots \hat A_{a_{l}} \cdot (1,1,1)^t}{q_l(\alpha)} \\
&=& \frac{\#\Omega^{(G^n(\alpha))}_{l-n}}{q_l(\alpha)}\sim\frac{q_{l-n}(G^n(\alpha))}{q_l(\alpha)}\sim\frac{1}{q_n(\alpha)}.
\end{eqnarray*}

Let $l\to \infty$, we conclude that  
\begin{equation}\label{uni-control}
\NN_{\alpha,\lambda}(B_{w})\sim\frac{1}{q_n(\alpha)}.
\end{equation} 
Combine \eqref{uniform-control} and \eqref{uni-control}, we get \eqref{sim-mu-NN}.
\end{proof}

\noindent{\bf Proof of Theorem \ref{main} (i).}\  
 By Theorem \ref{exact-hp-dim}, $\mu_\alpha$ is   exact upper and lower dimensional. Since $\pi_\alpha$ is bi-Lipschitz,  and $(\pi_\alpha)_\ast(\mu_\alpha)\asymp \NN_{\alpha,\lambda}$, $\NN_{\alpha,\lambda}$ has the same property.
 
The statement that  $\NN_{\alpha,\lambda}$ is exact-dimentional if and only if $d_H(\alpha,\lambda)=d_P(\alpha,\lambda)$  follows directly  from the definition and the fact  that  $\NN_{\alpha,\lambda}$ is exact upper and lower dimensional.
\hfill $\Box$


\subsection{Non exact-dimentional DOS}

In this subsection, we prove Theorem \ref{main} (ii), i.e.,  we construct a Sturm Hamiltonian $H_{\alpha,\lambda,\theta}$ such that the related DOS is not exact-dimensional. We roughly describe the idea as follows:  by \cite{Q},   for $\alpha_\kappa=[0;\kappa,\kappa,\cdots]$,  $\NN_{\alpha_\kappa,\lambda}$ is exact-dimensional with dimension $d(\kappa,\lambda)$, and  $d(1,\lambda)<d(2,\lambda)$ for $\lambda$ large enough. We will construct a frequency $\alpha=[0;a_1,a_2,\cdots]\in \mathscr{B}_2$ with $a_1a_2\cdots=1^{t_1}2^{\tau_1}1^{t_2}2^{\tau_2}\cdots$ such that $t_n\gg \tau_{n-1}$ and $\tau_n\gg t_n$. For such a frequency $\alpha$, $\NN_{\alpha,\lambda}$ is, in certain sense,  a concatenation of  $\NN_{\alpha_1,\lambda}$ and $\NN_{\alpha_2,\lambda}$ in  an alternative way. We can show that  for $x$ in a subset of positive $\NN_{\alpha,\lambda}$ measure,  the quantity   $\log \NN_{\alpha,\lambda}(B(x,r))/\log r$ oscillates between $d(1,\lambda)$ and $d(2,\lambda),$ which means that the local dimension does not exist at $x$. In particular,  $\NN_{\alpha,\lambda}$ can not  be exact-dimensional.
 
 \begin{prop}\label{example-not-eq}
 There exist  $\alpha\in \mathscr{B}_2$  and $\lambda_0>20$ such that 
$
d_H(\alpha,\lambda)<d_P(\alpha,\lambda)
$
for any $\lambda\ge \lambda_0$.
\end{prop}

\begin{proof}
  By Lemma \ref{a-a-q-n} (ii), for any $\beta\in\mathscr{B}_2,$
\begin{equation}\label{8.1}
q_n(\beta)\le 3^n.
\end{equation}
Define $\Phi^\beta$ by \eqref{def-Phi} and assume $\mu_\beta$ is a Gibbs-like measure of $\Phi^\beta.$ By \eqref{uniform-control}, there exists an absolute constant $C>1$   such that for any $w\in \Omega^{(\beta)}_{n},$
\begin{equation}\label{8.2}
\frac{1}{Cq_n(\beta)}\le \mu_\beta([w]^\beta) \le \frac{C}{q_n(\beta)}.
\end{equation}

 By Remark \ref{rem2.4}, $\Omega_{\alpha_\kappa}$ is a subshift of finite type with alphabet $\A_\kappa$ and incidence matrix $A_{\kappa\kappa}$ for $\kappa=1,2.$  Define $\Psi^\kappa=\{\psi^\kappa_n: n\ge 1\}$ on $\Omega_{\alpha_\kappa}$  by 
 $$
 \psi^\kappa_n(x):=\ln |B_{\varpi^{x_1}x|_n}|,
 $$
 where $\varpi^{x_1}\in \Omega^{(\alpha_\kappa)}_5$ is some fixed word such that $\varpi^{x_1}x_1$ admissible (see \cite{Q} (4.3)).  
 By \cite{Q} (Theorem 10 and eq. (5.7)), there exist an ergodic measure $\nu_{\kappa}$ supported on $\Omega_{\alpha_\kappa}$ and a constant $C_\kappa>1$ such that (warn  that in \cite{Q}, $\alpha_\kappa$ is defined as $\alpha_\kappa:=[\kappa;\kappa,\kappa,\cdots]$ )
 \begin{equation}\label{kappa-case}
 \dim_H \nu_{\kappa}=\frac{\log \alpha_\kappa}{\Psi_\ast^\kappa (\nu_{\kappa})}=:d(\kappa,\lambda)\  \text{ and }\  C_\kappa^{-1} \alpha_\kappa^{|w|}\le  \nu_{\kappa}([w]_{\alpha_\kappa})\le C_\kappa \alpha_\kappa^{|w|},
 \end{equation}
 where $\Psi_\ast^\kappa (\nu_{\kappa}):=\lim_{n\to\infty} \frac{1}{n}\int_{\Omega_{\alpha_\kappa}}\psi^\kappa_n d\nu_\kappa$.
 
 By \cite{Q} Proposition 6 and Remark 7, there exists $\lambda_0>20$ such that 
 $
 d(1,\lambda)<d(2,\lambda)
 $ 
 when $\lambda\ge \lambda_0$.
Since $\nu_{\kappa}$ is   ergodic, for $ \nu_{\kappa}$-a.e. $x\in \Omega_{\alpha_\kappa}$, 
 \begin{equation}\label{8.3}
\frac{\psi_n^\kappa(x)}{n}\to \Psi_\ast^\kappa( \nu_{\kappa}).
\end{equation}

Fix $\kappa\in \{1,2\}$, $n\in \N$ and $\epsilon>0$, define 
\begin{equation}\label{8.4}
F_\kappa(n,\epsilon)=\{w\in \Omega_{\alpha_\kappa,n}: \left|\frac{\psi_n^\kappa(x)}{n}- \Psi_\ast^\kappa( \nu_{\kappa})\right|\le \epsilon \text{ for any } x\in [w]_{\alpha_\kappa}  \}.
\end{equation}
Since $\psi_n^\kappa(x)=\psi_n^\kappa(y)$ if $x|_n=y|_n,$  we can replace ``any"
by ``some" in the definition above.
For each $e\in \A_{\kappa}$, define 
$$
F_\kappa(e,n,\epsilon):= F_\kappa(n,\epsilon)\cap \Xi_{e,n-1}(\alpha_\kappa).
$$

\noindent {\bf Claim:} For any $m\in\N$, there exists $l_m\in \N$ such that for any $\kappa\in\{1,2\}$, any $e\in \A_{\kappa}$ and any $n\ge l_m,$
\begin{equation*}
\frac{\#F_\kappa(e,n,2^{-m})}{\#\Xi_{e,n-1}(\alpha_\kappa)}\ge 1-C^{-2}2^{-m}.
\end{equation*}

\noindent $\lhd$
Fix any $e\in \A_{\kappa}$. By \eqref{8.3}, $\nu_{\kappa}(A_n(e))\to \nu_{\kappa}([e]_{\alpha_\kappa})$ as $n\to\infty$, where 
$$
A_n(e):=\bigcup_{w\in F_\kappa(e,n,2^{-m})}[w]_{\alpha_\kappa}.
$$
Write $\eta_m:=C_\kappa^{-2}C^{-2}2^{-m}$. Assume $l_m(e)\in \N$ is such that for any $n\ge l_m(e)$,
\begin{equation}\label{eq-A-n-e}
\nu_{\kappa}(A_n(e))\ge (1-\eta_m)\nu_{\kappa}([e]_{\alpha_\kappa}).
\end{equation}
Define $l_m:= \max\{l_m(e): e\in \A_{1}\cup\A_2\}$. Fix any $e\in \A_{\kappa}$ and  $n\ge l_m$, by \eqref{kappa-case}, 
\begin{eqnarray*}
\nu_{\kappa}([e]_{\alpha_\kappa}\setminus A_n(e))\ge (\#\Xi_{e,n-1}(\alpha_\kappa)-\#F_\kappa(e,n,2^{-m}))C_\kappa^{-1}\alpha_\kappa^{n}.
\end{eqnarray*}
By \eqref{eq-A-n-e} and \eqref{kappa-case},  we have
\begin{eqnarray*}
\nu_{\kappa}([e]_{\alpha_\kappa}\setminus A_n(e))\le \eta_m\ \nu_{\kappa}([e]_{\alpha_\kappa})\le \eta_m\#\Xi_{e,n-1}(\alpha_\kappa)C_\kappa\alpha_\kappa^{n}.
\end{eqnarray*}
Consequently, we have 
\begin{equation*}
\frac{\#\Xi_{e,n-1}(\alpha_\kappa)-\#F_\kappa(e,n,2^{-m})}{\#\Xi_{e,n-1}(\alpha_\kappa)}\le C^{-2}2^{-m},
\end{equation*}
which implies the claim.
\hfill $\rhd$

For any $\beta\in \mathscr{B}_2$, define $\Psi^\beta$ by \eqref{def-Psi}. By \eqref{band-length}, for any $x\in \Omega^{(\beta)},$
\begin{equation}\label{est-psi-n-1}
2\log 2+ n\log c\le \psi_n^\beta(x)\le 2\log 2-n\log 2,
\end{equation}
where $0<c<1$ only depends on $\lambda$.
By \cite{Q} (4.4), there exists constant $d(\lambda)>0$ such that for any   $ \kappa\in \{1,2\}$, $n\in\N$ and $y\in \Omega_{\alpha_\kappa},$
\begin{equation}\label{est-psi-n-2}
-d(\lambda)(n+6)\le  \psi_n^\kappa(y)\le -n\log 2.
\end{equation}

   Now we define two sequences $\{t_n: n\ge 1\}$ and $\{\tau_n: n\ge 1\}$ inductively as follows.  Define $t_1=\tau_1:=l_1.$
Assume $t_1,\cdots,t_{n-1}, \tau_1,\cdots, \tau_{n-1}$ have been defined. Write $\widehat T_{n-1}:=\sum_{j=1}^{n-1}(t_j+\tau_j)$ with convention $\widehat T_0=0$ and choose  $t_n\ge l_n$ such that
\begin{equation}\label{8.5}
\begin{cases}
\widehat T_{n-1}\log 3+\log q_{t_n}(\alpha_1)\le (1+\frac{1}{n})\log q_{t_n}(\alpha_1),\\
|\psi^\beta_{\widehat T_{n-1}}|_{\max}+ \psi_{t_n}^1(x)\le (1-\frac{1}{n})\psi_{t_n}^1(x), \ (\forall \beta\in \mathscr{B}_2,\ \forall x\in \Omega_{\alpha_1}).
\end{cases}
\end{equation}
 By \eqref{est-psi-n-1}, \eqref{est-psi-n-2} and the fact that $q_n(\alpha_1)\sim \alpha_1^{-n}$, such $t_n$ exists.
 Write $T_n:=\widehat T_{n-1}+ t_n$ and choose $\tau_n\ge l_n$ such that
\begin{equation}\label{8.6}
-|\psi^\beta_{T_{n}}|_{\min}+ \psi^2_{\tau_n}(y)\ge (1+\frac{1}{n})\psi^2_{\tau_n}(y), \ (\forall \beta\in \mathscr{B}_2,\ \forall y\in \Omega_{\alpha_2}).
\end{equation}
 By \eqref{est-psi-n-1} and \eqref{est-psi-n-2}, such $\tau_n$ exists.

Define $\alpha:=[0;a_1,a_2,\cdots]\in \mathscr{B}_2$ with $a_1a_2\cdots=1^{t_1}2^{\tau_1}1^{t_2}2^{\tau_2}\cdots$. Assume $\mu_\alpha$ is a Gibbs-like measure of $\Phi^\alpha. $ Let us show that 
$$
\dim_H\mu_\alpha\le d(1,\lambda)\ \ \ \text{ and }\ \ \  \dim_P\mu_\alpha\ge d(2,\lambda).
$$

Recall that $\mathcal{A}_1=\{(I,1)_1,(I,2)_1,II_1,(III,1)_1\}=\{e_{1,1},e_{1,2},e_{1,3},e_{1,4}\}$.
Define a Cantor subset $\mathscr{C}$ of $\Omega^{(\alpha)}$ as follows. Write $\epsilon_m=2^{-m}$ and define 
$$
\begin{cases}
W_1&:=\{I u : u\in F_1(e_{1,3},t_1,\epsilon_1)\},\\
\widehat W_1&:=\{uv\in\Omega^{(\alpha)}_{\widehat T_1}: u\in W_1, \  v\in F_2(\tau_1,\epsilon_1)\}.
\end{cases}
$$
Assume $W_{n-1}, \widehat W_{n-1}$ have been defined, define $W_n,\widehat W_n$ as
$$
\begin{cases}
W_n&:=\{uv\in \Omega^{(\alpha)}_{T_n}: u\in \widehat W_{n-1}, \   v\in F_1(t_n,\epsilon_n)\},\\
\widehat W_n&:=\{uv\in\Omega^{(\alpha)}_{\widehat T_n}: u\in W_n,\ v\in F_2(\tau_n,\epsilon_n)\}.
\end{cases}
$$
Define 
$
\mathscr{C}_n:=\bigcup_{w\in W_n} [w]^\alpha$ and $\widehat{\mathscr{C}}_n:=\bigcup_{w\in \widehat W_n} [w]^\alpha.
$
It is seen that $\mathscr{C}_{n+1}\subset \widehat{\mathscr{C}}_n\subset \mathscr{C}_n$. Define 
$
\mathscr{C}:=\lim_{n\to\infty}\mathscr{C}_n.
$
 Let us show that  $\mu_\alpha(\mathscr{C})>0.$
 
  At first,  $W_1$ is nonempty by the Claim. Then by \eqref{8.2}, 
 $
 \mu_\alpha (\mathscr{C}_1)>0.
 $
 Next, we compare $ \mu_\alpha (\widehat{\mathscr{C}}_1)$ and $ \mu_\alpha (\mathscr{C}_1).$ By \eqref{8.2} and the Claim, we have 
 \begin{eqnarray*}
 \mu_\alpha(\mathscr{C}_1\setminus \widehat{\mathscr{C}}_1) &=& \sum_{w\in W_1} \sum_{ e\in \A_{2}\atop w\to e }\sum_{u\in \Xi_{e,\tau_1-1}(\alpha_2)\setminus F_2(e,\tau_1,\epsilon_1)}\mu_\alpha([wu]^\alpha) \\
  &\le& \sum_{w\in W_1} \sum_{ e\in \A_{2}\atop w\to e }\frac{C(\#\Xi_{e,\tau_1-1}(\alpha_2)-\#F_2(e,\tau_1,\epsilon_1))}{q_{\widehat T_1}(\alpha)}\\
  &\le& C^{-1}2^{-1}\sum_{w\in W_1} \sum_{ e\in \A_{2}\atop w\to e }\frac{\#\Xi_{e,\tau_1-1}(\alpha_2)}{q_{\widehat T_1}(\alpha)}\\
&\le&2^{-1}  \sum_{w\in W_1} \sum_{ e\in \A_{2}\atop w\to e }\sum_{u\in \Xi_{e,\tau_1-1}(\alpha_2)}\mu_\alpha([wu]^\alpha)=2^{-1}\mu_\alpha(\mathscr{C}_1).
  \end{eqnarray*}
  Or equivalently, we have $\mu_\alpha(\widehat{\mathscr{C}}_1)\ge (1-2^{-1})\mu_\alpha(\mathscr{C}_1).$ By essentially the same proof as above, we can show that for any $n\ge2,$
  $$
  \mu_\alpha({\mathscr{C}}_n)\ge (1-2^{-n})\mu_\alpha(\widehat{\mathscr{C}}_{n-1})\ \  \ \text{ and }\ \ \ \mu_\alpha(\widehat{\mathscr{C}}_n)\ge (1-2^{-n})\mu_\alpha(\mathscr{C}_n).
  $$
  As  a consequence, we get 
  $$
  \mu_\alpha(\mathscr{C})\ge\mu_\alpha(\mathscr{C}_1) (1-2^{-1})\prod_{n\ge 2}(1-2^{-n})^2>0.
  $$

Take any $x\in \mathscr{C}$ and write $x|_{\widehat T_n}=Iu_1v_1\cdots u_{n}v_{n}$ such that  $|u_i|=t_i$ and $|v_i|=\tau_i.$ Then  $u_i\in F_1(t_i,\epsilon_i)$, $v_i\in F_2(\tau_i,\epsilon_i)$ and   $x|_{ T_n}=x|_{\widehat T_{n-1}}u_n$,  $x|_{\widehat T_n}=x|_{T_n}v_n.$ By the definition of $\alpha$, 
$G^{\widehat T_{n-1}}(\alpha)=[0;\underbrace{1,\cdots,1}_{t_n},\cdots], $ hence $q_{t_n}(G^{\widehat T_{n-1}}\alpha)=q_{t_n}(\alpha_1).$  By \eqref{8.2}, Lemma \ref{a-a-q-n} and \eqref{8.1}, we have 
 \begin{eqnarray*}
 \mu_\alpha([x|_{T_n}]^\alpha)\sim \frac{1}{ q_{T_n}(\alpha)} \sim \frac{1}{q_{\widehat T_{n-1}}(\alpha)q_{t_n}(G^{\widehat T_{n-1}}\alpha)}\ge \frac{1}{3^{\widehat T_{n-1}} q_{t_n}(\alpha_1)}.
\end{eqnarray*}
 By Lemma \ref{lem-weak-gibbs} (ii), Lemma \ref{bdvar-Psi} and \eqref{def-Psi}, 
\begin{eqnarray*}
&&{\rm diam} ([x|_{T_n}]^\alpha)\sim_\lambda\  \exp(\psi_{T_n}^\alpha(x))=|B_{x|_{T_n}}|=\frac{|B_{x|_{T_n}}|}{|B_{x|_{\widehat T_{n-1}}}|}\cdot |B_{x|_{\widehat T_{n-1}}}|\\
&=& \frac{|B_{x|_{\widehat T_{n-1}}u_n}|}{|B_{x|_{\widehat T_{n-1}}}|}\cdot\exp(\psi_{\widehat T_{n-1}}^\alpha(x)).
\end{eqnarray*}

Now take any $y^{n}\in [u_n]_{\alpha_1},$ we have 
\begin{eqnarray*}
\exp(\psi_{t_n}^1(y^{n}))&=&|B_{\varpi^{y^n_1}y^{n}|_{t_n}}|=|B_{\varpi^{y^n_1}u_n}|=\frac{|B_{\varpi^{y^n_1}u_n}|}{|B_{\varpi^{y^n_1}}|}|B_{\varpi^{y^n_1}}|.
\end{eqnarray*}
Since $\varpi^{y^n_1}\in \Omega_5^{(\alpha_1)},$ it  only has finitely many choices, combine with Theorem \ref{bco}, we have 
$$
{\rm diam} ([x|_{T_n}]^\alpha)\sim_\lambda\ \exp(\psi_{t_n}^1(y^n))\exp(\psi_{\widehat T_{n-1}}^\alpha(x)).
$$
Recall that $u_n\in F_1(t_n,\epsilon_n)$, thus  by \eqref{8.5}, \eqref{8.4}, \eqref{kappa-case} and the fact that $q_n(\alpha_1)\sim \alpha_1^{-n}$, we have 
\begin{eqnarray*}
\underline{d}_{\mu_\alpha}(x)
&\le &\liminf_{n\to\infty}\frac{\log\mu_\alpha([x|_{T_n}]^\alpha)}{\log {\rm diam} ([x|_{T_n}]^\alpha)}\\
&\le &\liminf_{n\to\infty}\frac{O(1)+\widehat T_{n-1}\log 3+\log q_{t_n}(\alpha_1)}{-\left(O(1)+\psi_{\widehat T_{n-1}}^\alpha(x)+ \psi_{t_n}^1(y^n)\right)}\\
&\le &\liminf_{n\to\infty}\frac{ (1+\frac{1}{n})\log q_{t_n}(\alpha_1)}{-(1-\frac{1}{n})\psi_{t_n}^1(y^n)}=d(1,\lambda).
\end{eqnarray*}

On the other hand, notice that $G^{T_{n}}(\alpha)=[0;\underbrace{2,\cdots,2}_{\tau_n},\cdots], $ thus $q_{\tau_n}(G^{T_n}(\alpha))=q_{\tau_n}(\alpha_2)$. Hence by \eqref{8.2} and  Lemma \ref{a-a-q-n}, we have 
 \begin{eqnarray*}
 \mu_\alpha([x|_{\widehat T_n}]^\alpha)\sim \frac{1}{q_{\widehat T_n}(\alpha)} \sim \frac{1}{q_{T_{n}}(\alpha)q_{\tau_n}(G^{T_n}(\alpha))}\le \frac{1}{ q_{\tau_n}(\alpha_2)}.
\end{eqnarray*}

Take any $z^{n}\in [v_n]_{\alpha_2},$ by similar argument as above, we also have 
\begin{eqnarray*}
{\rm diam} ([x|_{\widehat T_n}]^\alpha)\sim_\lambda\ \exp(\psi_{\tau_n}^2(z^{n}))\exp(\psi_{ T_{n}}^\alpha(x)).
\end{eqnarray*}
 Recall that $v_n\in F_2(\tau_n,\epsilon_n)$, thus  by \eqref{8.6}, \eqref{8.4}, \eqref{kappa-case} and the fact that $q_n(\alpha_2)\sim \alpha_2^{-n}$, we have
\begin{eqnarray*}
\overline{d}_{\mu_\alpha}(x)
&\ge &\limsup_{n\to\infty}\frac{\log\mu_\alpha([x|_{\widehat T_n}]^\alpha)}{\log {\rm diam} ([x|_{\widehat T_n}]^\alpha)}\\
&\ge &\limsup_{n\to\infty}\frac{O(1)+\log q_{\tau_n}(\alpha_2)}{-\left(O(1)+\psi_{ T_{n}}^\alpha(x)+ \psi_{\tau_n}^2(z^{n})\right)}\\
&\ge &\limsup_{n\to\infty}\frac{ \log q_{\tau_n}(\alpha_2)}{-(1+\frac{1}{n})\psi_{\tau_n}^2(z^{n})}=d(2,\lambda).
\end{eqnarray*}
As a result, for any $x\in \mathscr{C},$ we have 
$$
\underline{d}_{\mu_\alpha}(x)\le d(1,\lambda) \ \ \ \text{ and }\ \ \ \overline{d}_{\mu_\alpha}(x)\ge d(2,\lambda).
$$
 Since $\mu_\alpha(\mathscr{C})>0$ and $\mu_\alpha$ is exact upper and lower dimensional by Theorem \ref{exact-hp-dim},  we conclude that 
 $$
 \dim_H\mu_\alpha\le d(1,\lambda)<d(2,\lambda)\le \dim_P\mu_\alpha.
 $$
 Now by   Proposition \ref{strong-equiv}, we have 
$
(\pi_\alpha)_\ast(\mu_\alpha) \asymp \NN_{\alpha,\lambda}.
$
Thus   
$$d_H(\alpha,\lambda)=\dim_H\mu_\alpha<\dim_P\mu_\alpha=d_P(\alpha,\lambda).
$$
and the result follows.
 \end{proof}
 
 \noindent {\bf Proof of Theorem \ref{main} (ii).}\ 
 For $\alpha$ constructed in Proposition \ref{example-not-eq},  $\NN_{\alpha,\lambda}$ is not exact-dimensional, since  its Hausdorff and packing dimensions do not coincide.
 \hfill $\Box$

 \smallskip

\noindent{\bf Acknowledgement}. The author 
 was supported by the National Natural Science Foundation of China,  No. 11371055 and No. 11431007.


 
\end{document}